\documentclass{amsart}
\usepackage{latexsym}
\usepackage{amsmath}
\usepackage{mathtools}
\usepackage{mathrsfs}
\usepackage{amssymb}
\usepackage[all,cmtip,2cell]{xy}
\UseTwocells

\usepackage{comment}

\usepackage{autobreak}
\usepackage{enumitem}
\usepackage[hypertexnames=false]{hyperref} 
\usepackage[capitalize]{cleveref}
\crefname{equation}{}{}
\crefname{enumi}{}{}

\usepackage[usenames]{xcolor}
\definecolor{refkey}{rgb}{0.6, 0.7, 0.4}
\definecolor{labelkey}{rgb}{0, 0.7, 0.5}
\newcommand{\enabletodonotes}{}

\usepackage[\ifdefined\enabletodonotes\else disable\fi]{todonotes}

\usepackage{tikz}
\usetikzlibrary{positioning}
\usetikzlibrary{arrows.meta}

\numberwithin{equation}{section}
\newtheorem{thm}{Theorem}[section]
\newtheorem{prop}[thm]{Proposition}

\newtheorem{lem}[thm]{Lemma}
\newtheorem{claim}[thm]{Claim}
\theoremstyle{definition}
\newtheorem{defn}[thm]{Definition}

\theoremstyle{remark}
\newtheorem{rem}[thm]{Remark}
\newtheorem{ex}[thm]{Example}

\newcommand{\R}{{\mathbb R}}
\newcommand{\F}{{\mathcal F}}
\newcommand{\D}{{\mathcal D}}
\newcommand{\e}{\varepsilon}

\newcommand{\im}{\operatorname{Im}}

\newcommand{\mapright}[1]{%
 \smash{\mathop{%
  \hbox to 1cm{\rightarrowfill}}\limits_{#1}}}
\newcommand{\maprightd}[2]{%
 \smash{\mathop{%
  \hbox to 0.5cm{\rightarrowfill}}\limits^{#1}\limits_{#2}}}
\newcommand{\mapleft}[1]{%
 \smash{\mathop{%
  \hbox to 1cm{\leftarrowfill}}\limits_{#1}}}
\newcommand{\mapleftu}[1]{%
 \smash{\mathop{%
  \hbox to 0.8cm{\leftarrowfill}}\limits^{#1}}}
\newcommand{\maprightu}[1]{%
 \smash{\mathop{%
  \hbox to 1cm{\rightarrowfill}}\limits^{#1}}}
\newcommand{\maprightud}[2]{%
 \smash{\mathop{%
  \hbox to 1cm{\rightarrowfill}}\limits^{#1}_{#2}}}
\newcommand{\mapleftud}[2]{%
 \smash{\mathop{%
  \hbox to 1cm{\leftarrowfill}}\limits^{#1}_{#2}}}

\newcommand{\ad}{\mathrm{ad}}
\newcommand{\colim}{\mathop{\mathrm{colim}}}
\newcommand{\ev}{\mathrm{ev}}
\newcommand{\id}{\mathrm{id}}
\newcommand{\pr}{\mathrm{pr}}
\newcommand{\vol}{\mathrm{vol}}
\newcommand{\Iso}{\mathrm{Iso}}

\author[K. Kuribayashi]{Katsuhiko Kuribayashi}
\address{%
  Department of Mathematics,
  Faculty of Science,
  Shinshu University,
  Matsumoto, Nagano 390-8621, Japan
}
\email{kurimath@shinshu-u.ac.jp}

\author[K. Sakai]{Keiichi Sakai}
\address{%
  Department of Mathematics,
  Faculty of Science,
  Shinshu University,
  Matsumoto, Nagano 390-8621, Japan
}
\email{sakaik@shinshu-u.ac.jp}

\author[Y. Shiobara]{Yusuke Shiobara}
\address{%
  Department of Science and Technology,
  Graduate School of Medicine, Science and Technology,
  Shinshu University,
  Matsumoto, Nagano 390-8621, Japan
}
\email{24hs603d@shinshu-u.ac.jp}

\title[Towards Riemannian diffeology]{Towards Riemannian diffeology 
}

\thanks{The first author was partially supported by JSPS KAKENHI Grant Number JP21H00982. The second author was partially supported by JSPS KAKENHI Grant Numbers JP20K03608, JP21H00982 and JP25K06972. The third author was supported by the Sasakawa Scientific Research Grant 2025-2017 from The Japan Science Society. }

\begin{document}
\begin{abstract}
We introduce a framework for Riemannian diffeology. 
To this end, we use the tangent functor in the sense of Blohmann and one of the options of a metric on a diffeological space in the sense of Iglesias-Zemmour.
As a consequence, the category consisting of weak Riemannian diffeological spaces and isometries is established. 
With a technical condition for a definite weak Riemannian metric, we show that the pseudodistance induced by the metric is indeed a distance.
As examples of weak Riemannian diffeological spaces, an adjunction space of manifolds, a space of smooth maps and the mixed one are considered.  
\end{abstract}

\maketitle

\section{Introduction}\label{sect:Introduction} 
Diffeology \cite{So, PIZ12} provides a natural generalization of differential topology and geometry. 
The de Rham theory \cite{Haraguchi, H-V-C, PIZ85,  PIZ_deRham, K2020, M23, G-PIZ}, sheaf theory \cite{N-A, KWW}, infinite dimensional geometry for partial differential equations \cite{GMW} and (abstract) homotopy theory \cite{C-W, Kihara1, Kihara, K2024, S-Y-H} have also been developed in the diffeological setting. 
Moreover, categorical comparisons of diffeology with other smooth and topological structures are made in \cite{B-H, C-S-W, S}.
However, it is hard to say that the development of Riemannian notions of diffeological spaces is sufficient.

In \cite{PIZ23}, Iglesias-Zemmour has introduced a notion of Riemannian metrics in diffeology. 
The definition of the definiteness of the metric on a diffeological space is described with differential $1$-forms on the given space; see \cite[Page 3]{PIZ23}. 
Making a quote from \cite[Page 227]{PIZ25}, 
``It is not clear what definition is the best, for many examples built with manifolds and spaces of smooth maps they do coincide. 
But they may differ in general and, depending on the problem, one must choose one or the other.''

This article introduces a weak Riemannian metric of a diffeological space and its definiteness using the tangent functor due to Blohmann \cite{B23, B24}. 
In particular, the colimit construction of the tangent functor works well in considering the metrics. 
As a consequence, we may deal with weak Riemannian metrics for a diffeological  adjunction space and the space of smooth maps in our framework simultaneously. 
A comparison between a weak Riemannian metric and the metric due to Iglesias-Zemmour is made in Propositions \ref{prop:one-to-one} and \ref{prop:IZ-Ours}. 
It is worthwhile mentioning that Goldammer and Welker introduce another definition of a Riemannian diffeological space in \cite{G-W} by using the tangent space in the sense of Vincent \cite{V}. 
While it is important to consider the relationship between two frameworks of Riemannian diffeology, we do not pursue the issue in this article.

Here is a summary of our main results. 
Let $X$ be a \textit{weak Riemannian diffeological space}, namely a diffeological space endowed with a weak Riemannian metric. 
The pseudodistance $d$ defined by a weak Riemannian metric on $X$ gives the topology $\mathcal{O}_d$ of $X$. 
Then, Theorem \ref{thm:D-top} allows us to conclude that the $D$-topology of $X$ is finer than $\mathcal{O}_d$. 
Theorem \ref{thm:condition_for_g_to_be_definite} asserts that the pseudodistance defined by a weak Riemannian metric on $X$ is indeed a distance if the metric is definite and the diffeology of $X$ is generated by a family of plots which \emph{separates points}; see Definition \ref{defn:points} for the separation condition. 
The necessity of the separateness for a diffeology is clarified in Example \ref{ex:Y}.

Theorem \ref{thm:adjunction_g} yields that an adjunction diffeological space obtained by attaching two definite weak Riemannian diffeological spaces admits again a definite weak Riemannian metric.
We also see that the space $C^\infty(M, N)$ of smooth maps admits a weak Riemannian metric, here $M$ is a closed manifold and $N$ is a weak Riemannian diffeological space; see Section \ref{ss:metric_on_MappingSapces}. Moreover, it turns out that the pseudodistance on $C^\infty(M, N)$ is a distance without assuming the separation condition; see Proposition \ref{prop:distance_on_Func}. 
Theorem \ref{thm:D'} discusses the definiteness of the metric of $C^\infty(M, N)$ endowed with appropriate subdiffeology (which is coarser than the functional diffeology in general).

These constructions are mixed. 
In fact, we obtain a fascinating example of a definite weak Riemannian diffeological space.

\begin{ex}[see Example \ref{ex:mapping_spaces_pushout} and Proposition \ref{prop:an_extension} for a more general setting]\label{ex:theFirstExample}
Let $M$ be a closed orientable manifold and $(N, g_N)$ a Riemannian manifold. The diffeological adjunction space 
\[
 C^\infty(M, N)\coprod_NC^\infty(M, N)
\] 
obtained by the section $N \to C^\infty(M, N)$ of the evaluation map admits a definite weak Riemannian metric $\widetilde{g}$ for which 
\[
 \iota^*(\widetilde{g}) = (\int_M \text{vol}_M) \times g_N,
\]
where the left-hand side is the pullback of the metric  $\widetilde{g}$ by the canonical injection $\iota \colon N \to  C^\infty(M, N)\coprod_NC^\infty(M, N)$ and $\text{vol}_M$ denotes the volume form of $M$; see Definition \ref{defn:Pullback} for the pullback of a metric. 
\end{ex}

By applying Proposition \ref{prop:An_induction}, which is a result on the pullback of a metric along an induction, and a general result on a weak Riemannian metric on a mapping space, we investigate one of the special examples of mapping spaces, \emph{free loop spaces}.
Proposition~\ref{prop:concatenation_preserves_metric} shows that the \emph{concatenation map}, that is important in the study of \emph{string topology} \cite{C-S}, preserves the metrics on the loop spaces. 
This result suggests that our weak Riemannian metric might be well-suited to string topology; see Section \ref{ss:free_loop_space}. 

The rest of this article is organized as follows. 
Section \ref{sect:Preliminaries} recalls the definition of a diffeological space and gives its examples. 
In Section \ref{sect:assertions}, we introduce a weak Riemannian metric, the (pseudo-)distance on a diffeological space and the category of weak Riemannian diffeological spaces containing (possibly infinite dimensional) manifolds with metrics.
As mentioned above, the metric is related to that in the sense of Iglesias-Zemmour.

Sections \ref{sect:Adjunction_spaces} and \ref{section:MappingSpaces} address metrics on a diffeological adjunction space and a diffeological mapping space, respectively. 
Section \ref{sect:warped_products} is devoted to considering the warped product in diffeology.  
In Section \ref{sect:OpenQ}, we propose problems on a comparison of diffeologies of an infinite-dimensional manifold, the subdiffeology of the functional diffeology of a diffeological mapping space discussed in Section \ref{section:MappingSpaces} and a relationship between Riemannian orbifolds and weak Riemannian diffeological spaces. 

\subsection{Future work and perspective} 
The category of diffeological spaces contains manifolds and other spaces with smooth structures; see Example \ref{ex:manifolds_subsets_products} and Section \ref{subsection:IDMfd} below.
Therefore, diffeology together with the notion of Riemannian metrics (Riemannian diffeology) produces many issues on smooth spaces. 
Here are some of them. 
\begin{enumerate}[wide,label=\roman*)]
\item
	With the framework of Riemannian metrics, we may consider geodesic calculus in diffeology as seen in \cite{PIZ25}.
\item
	We expect a new development of convergence theory of smooth spaces with concepts such as connections, curvatures and the Gromov-Hausdorff distance if the notion of metrics of diffeological spaces combines with \textit{elastic diffeology}
\footnote{See \cite{B23, C-C14, C-C15} for Lie brackets of vector fields which are defined in a category with a tangent structure in the sense of Rosick\'y \cite{R}. We refer the reader to  \cite{C-C17} for connections in tangent categories.}.
\item
	Topological data analysis (TDA) begins usually by putting a data set into an appropriate metric space, in particular an Euclidean space, and investigate shapes of the data. Thus, future work  includes moreover studying TDA by using a Riemannian diffeological space as the stage for a data set. 
\end{enumerate}

\section{Preliminaries}\label{sect:Preliminaries}
In this section, we recall examples of diffeological spaces and important fundamental constructions in diffeology which are used throughout this article.
We begin with the definition of a diffeological space.
A comprehensive reference for diffeology is the book by Iglesias-Zemmour \cite{PIZ12}.  
\begin{defn}
We call an open subset $U$ of a Euclidean space (of arbitrary, but finite dimension) a \textit{domain}.
A map from a domain to a set $X$ is called a \textit{parametrization} of $X$.
The domain of a parametrization $P$ of $X$ is denoted by $U_P$ or $\mathrm{dom}(P)$.

Let $X$ be a set. 
A set $\D$ of parametrizations of $X$ is a \textit{diffeology} of $X$ if the following three conditions hold: 
\begin{enumerate}[wide,leftmargin=0pt]
\item
	(Covering) Every constant map $U \to X$ for all domains $U$ is in $\D$;
\item
	(Compatibility) If $U \to X$ is in $\D$, then for any smooth map $V \to U$ from another domain $V$, the composite $V \to U \to X$ is also in $\D$; 
\item
	(Locality) If $U = \bigcup_i U_i$ is an open cover of a domain $U$ and $U \to X$ is a map such that each restriction $U_i \to X$ is in $\D$, then $U \to X$ is in $\D$. 
\end{enumerate}
An element of a diffeology $\D$ of $X$ is called a \textit{plot} of $X$.
A \textit{diffeological space} $(X, \D)$ consists of a set $X$ and a diffeology $\D$ of $X$.

For diffeological spaces  $(X, \D^X)$ and $(Y, \D^Y)$, a map $f\colon X \to Y$ is \textit{smooth} if the composite $f\circ P$ is in $\D^Y$ for each plot $P \in \D^X$.  
Then, we have the category $\mathsf{Diff}$ consisting of all diffeological spaces and smooth maps.
The isomorphisms in $\mathsf{Diff}$ are referred to as \emph{diffeomorphisms}.
Denote as $X\cong Y$ if $X,Y\in\mathsf{Diff}$ are diffeomorphic.
\end{defn}

Every set admits (possibly trivial) diffeologies;
the \emph{discrete (resp.\ indiscrete) diffeology} consisting of locally constant (resp.\ all the) parametrizations.
A typical example of a diffeological space is a manifold, as in the following Example~\ref{ex:manifolds_subsets_products} (\ref{item:standard_diffeology}).
We also have many other examples each of which is not a manifold in general.

\begin{ex}\label{ex:manifolds_subsets_products}
\begin{enumerate}[wide,leftmargin=0pt]
\item\label{item:standard_diffeology} 
	Let $M$ be a manifold of finite dimension.
	Then, the underlying set $M$ and the \textit{standard diffeology} $\D^M_{\text{std}}$ give rise to a diffeological space $(M, \D^M_{\text{std}})$, where $\D^M_{\text{std}}$ is defined to be the set of all smooth maps $U \to M$ from domains to $M$ in the usual sense.
	We have a functor $m \colon \mathsf{Mfd} \to \mathsf{Diff}$ from the category of manifolds defined by $m(M) = (M, \D^M_{\text{std}})$.
	This functor is fully faithful; see, for example, \cite[2.1 Example]{B-H}. 
\item\label{item:generating_family}
	Let $\mathcal{G}$ be a set of parametrizations of $X$.
	Then define $\langle\mathcal{G}\rangle$ to be the set of all parametrizations $P\colon U_P\to X$ satisfying the following condition:
	\begin{itemize}
	\item[]
		For any $r\in U_P$, there exists an open neighborhood $V_r$ of $r$ in $U_P$ such that $P|_{V_r}$ is a constant plot, or there exist $Q\in\mathcal{G}$ and a smooth map $f\colon V_r\to U_Q$ such that $P|_{V_r}=Q\circ f$.
	\end{itemize}
	Then $\langle\mathcal{G}\rangle$ is a diffeology of $X$, called the \emph{diffeology generated by $\mathcal{G}$}.

	Let $(X,\mathcal{D})$ be a diffeological space.
	A subset $\mathcal{G}\subset\mathcal{D}$ is a \emph{generating family} of $\mathcal{D}$ if $\langle\mathcal{G}\rangle=\mathcal{D}$. 
	Examples include generating families $\mathcal{G}_{\text{atlas}}$ and $\mathcal{G}_{\text{imm}}$ of the standard diffeology $\mathcal{D}^M_{\text{std}}$ of a manifold $M$ with an atlas $\{(V_\lambda, \varphi_\lambda)\}_{\lambda \in \Lambda}$ given respectively by $\mathcal{G}_{\text{atlas}}=\{i_\lambda\circ \varphi_{\lambda}^{-1} \colon \varphi_\lambda(V_{\lambda})\to M \}_{\lambda \in \Lambda}$ (where $i_\lambda \colon V_{\lambda}\to M$ is the inclusion), and the set of immersions $\mathcal{G}_{\text{imm}}=\{\text{immersions }f\colon U\looparrowright M\}$.
\item\label{item:sub_diff}
	For a diffeological space $(X, \D^X)$ and a subset $A$ of $X$, we define $\D^A_{\text{sub}}$ by 
	$\D^A_{\text{sub}}\coloneqq \{ P \colon U_P \to A \mid U_P \text{ is a domain and } i\circ P \in \D^X\}$,
	where $i \colon A \to X$ is the inclusion.
	Then, the set $\D^A_{\text{sub}}$ is a diffeology of $A$, which is called the \textit{sub-diffeology}.
	We call $(A, \D^A_{\text{sub}})$ a \textit{diffeological subspace} of $X$. 
\item\label{item:product_diff}
	Let $\{(X_i, \D_i) \}_{i \in I}$ be a family of diffeological spaces.
	Then, the product $\Pi_{i\in I}X_i$ has a diffeology $\D$, called the \textit{product diffeology}, defined to be the set of all parameterizations $P\colon U_P \to \Pi_{i\in I}X_i$ such that $\pi_i\circ p$ are plots of $X_i$ for each $i \in I$, where $\pi_i \colon \Pi_{i\in I}X_i \to X_i$ denotes the canonical projection.
\item\label{item:induction}
	More general, the \textit{initial diffeology} $\D^Y$ for maps $h_i \colon Y \to (X_i, \D_i)$ for $i \in I$ is defined by $\D^Y \coloneqq\{P\colon U_P \to Y \mid h_i\circ P \in \D_i \text{ for } i \in I\}$.
	This is the largest diffeology on $Y$ making all $h_i$ smooth.
	We call an injective map $j \colon X\to Y$ between diffeological spaces an \textit{induction} if $\mathcal{D}^X$ coincides with the initial diffeology for $j$.
	It is immediate that if $A$ is a diffeological subspace of $X$, then the inclusion $A \to X$ is an induction. 
\item\label{item:functional_diffeology}
	Let $(X, \D^X)$ and $(Y, \D^Y)$ be diffeological spaces.
	Let $C^\infty(X, Y)$ denote the set of all smooth maps from $X$ to $Y$.
	The \textit{functional diffeology} $\D_{\text{func}}$ is defined to be the set of parametrizations $P\colon U_P \to C^\infty(X, Y)$ whose adjoints $\ad(P) \colon U_P \times X \to Y$ are smooth.  
\item\label{item:final_diffeology}
	Let ${\mathcal F} \coloneqq\{f_i \colon Y_i \to X\}_{i\in I}$ be a set of maps from diffeological spaces $(Y_i, \D^{Y_i})$ ($i \in I$) to a set $X$. 
	Then, we can define a diffeology $\D^X$ of $X$ to be the set of parametrizations $P \colon U \to X$ satisfying the following condition;
	for $r \in U$,
	(i) there exists an open neighborhood $V_r$ of $r$ in $U$ such that $P|_{V_r}$ is constant, or
	(ii) for $i\in I$, there exists an open neighborhood $V_{r,i}$ of $r$ in $U$ and a plot $(P_i\colon V_{r,i}\to Y_i)\in \D^{Y_i}$ with $P|_{V_{r,i}} = f_i\circ P_i$.
	We call $\D^X$ the \textit{final diffeology} of $X$ with respect to ${\mathcal F}$. This is the smallest diffeology on $X$ making all $f_i$ smooth.

	Moreover, by definition, a surjective map $\pi\colon X\to Y$ between diffeological spaces is a \textit{subduction} if the diffeology of $Y$ is the final diffeology with respect to $\pi$.  

	We shall say that a smooth surjection $\pi\colon X\to Y$ is a \textit{local subduction} if for a point $x \in X$ and each plot $P \colon U_P \to Y$ with $P(0) = \pi(x)$,
	there exist an open neighborhood $W$ of $0$ in $U_P$ and a plot $Q \colon W\to X$ such that $Q(0) = x$ and $\pi\circ Q = P|_W$. 
\item\label{item:sum_diff}
	For a family of diffeological spaces $\{(X_i, \D_i) \}_{i \in I}$,
	the coproduct $\coprod_{i\in I}X_i$ has the final diffeology with respect to the set of canonical inclusions. The diffeology is called the \textit{sum diffeology}. 
\item\label{item:quotient_diff}
	Let $(X, \D)$ be a diffeological space with an equivalence relation $\sim$.
	Then, the final diffeology of $X/\mathord{\sim}$ with respect to the quotient map $q\colon X \to X/\mathord{\sim}$ is called the \textit{quotient diffeology}. 
	In particular $q$ is a subduction. 
\end{enumerate}
\end{ex}

We have a result concerning the construction (\ref{item:induction}).
\begin{prop}\label{prop:induction}\text{\em (\cite[1.36]{PIZ12})}
An induction $i \colon (A, \D^A) \to (X, \D^X)$ gives rise to a diffeomorphism $i\colon (A, \D^A)\stackrel{\cong}{\to} (\im(i), \D^{\im(i)}_{\text{\em sub}})$.
\end{prop}

The constructions (\ref{item:product_diff}) and (\ref{item:functional_diffeology}) above enable us to obtain an adjoint pair 
\[
 \text{-} \times X : \mathsf{Diff} \leftrightarrows \mathsf{Diff} : C^\infty(X, \text{-}).
\]
Moreover, the limit and colimit in $\mathsf{Diff}$ are constructed explicitly by (\ref{item:product_diff}) with (\ref{item:sub_diff}) and (\ref{item:sum_diff}) with (\ref{item:quotient_diff}), respectively,  via the forgetful functor from $\mathsf{Diff}$ to the category of sets; see also \cite[Section 2]{C-S-W}.
Thus, we have 
\begin{thm} \text{\em (\cite{B-H, PIZ85})}
The category $\mathsf{Diff}$ is complete, cocomplete and cartesian closed. 
\end{thm}

Another adjoint pair is given by the \emph{$D$-topology functor} and the \emph{continuous diffeology functor}
\begin{equation*}
 D\colon\mathsf{Diff}\leftrightarrows\mathsf{Top}:\!C, 
\end{equation*}
where $\mathsf{Top}$ is the category of topological spaces and continuous maps, and
\begin{itemize}
\item
	for $(X,\mathcal{D}^X)\in\mathsf{Diff}$,
	the topological space $D(X)$ consists of the underlying set $X$ and the set of open sets $\mathcal{O}$ defined by
	\begin{equation*}
	\mathcal{O}\coloneqq\{O\subset X\mid P^{-1}(O)\text{ is an open set of }U_P\text{ for each }P\in\mathcal{D}^X\},
	\end{equation*}
\item
	for $Y\in\mathsf{Top}$,
	the diffeological space $C(Y)$ consists of the underlying set $Y$ and the diffeology
	$\mathcal{D}^Y\coloneqq\{\text{continuous maps }P\colon U_P\to Y\}$.
\end{itemize}

\begin{ex}\label{ex:Dm}
The composite of the functor $\xymatrix@C15pt@R20pt{\mathsf{Mfd} \ar[r]^m & \mathsf{Diff} \ar[r]^D & \mathsf{Top}}$ coincides with the forgetful functor; see \cite[Section 3]{C-S-W}.   
We refer the reader to \cite[Proposition 3.3]{C-S-W}, \cite[Proposition 2.1]{S-Y-H} and \cite[Theorem 1.5]{Kihara} for other properties of the adjoint pair $D \dashv C$. 
\end{ex}

\begin{rem}\label{rem:The_quotient_map}
We see that $X \cong \colim_{P\in \D}U_P$ as a diffeological space; see \cite[Exercise 33]{PIZ12} and \cite[Proposition 2.7]{C-S-W}. 
Since the $D$-topology functor $D \colon \mathsf{Diff} \to \mathsf{Top}$ is the left adjoint, it follows that $D$ preserves colimits.
Therefore, natural maps give rise to homeomorphisms  
\[
 D(X) \cong D(\colim_{P\in \D}U_P) \cong \colim_{P\in \D}D(U_P) \cong  \colim_{P\in \D}U_P, 
\]
where the last space is endowed with the quotient topology; see also Example \ref{ex:Dm} for the fact that $D(U_P)=U_P$.  
\end{rem}

\section{A Riemannian diffeological space}\label{sect:assertions}
\subsection{Diffeological Riemannian metrics}
Let $\mathsf{Euc}$ denote the category consisting of domains and smooth maps.
Let $\widehat{T} \colon \mathsf{Euc} \to  \mathsf{Euc}$ be the functor defined by 
\[
 \widehat{T}(U) \coloneqq U \times \R^{\dim U}
\] 
and  
$\mathcal{Y} \colon \mathsf{Euc} \to  \mathsf{Diff}$ the Yoneda functor. 
We recall the tangent functor $T \colon \mathsf{Diff} \to \mathsf{Diff}$ from the category of diffeological spaces to itself in the sense of Blohmann \cite{B23, B24}, 
which is the left Kan extension 
$\mathbb{L}\mathcal{Y}\widehat{T}\coloneqq\text{Lan}_{\mathcal{Y}}\mathcal{Y}\widehat{T}$ of the functor 
$\mathcal{Y}\widehat{T} \colon \mathsf{Euc} \to  \mathsf{Diff}$ along the Yoneda functor $\mathcal{Y}$; see also \cite{M}.   
For a diffeological space $(X, \D)$, it turns out that 
\[
 T(X) = \colim_{P \in \D} (U_P\times \R^{\dim U_P}).
\] 
Here, the diffeology $\D$ is regarded as a category whose objects are plots of $X$ and whose morphisms are smooth maps $h\colon U_P \to U_Q$ with $Q\circ h = P$.

By using the functor $\widehat{T}_2 \colon \mathsf{Euc} \to \mathsf{Diff}$ defined by 
$\widehat{T}_2(U) \coloneqq \widehat{T}(U)\times_U\widehat{T}(U)=U \times \R^{\dim U} \times  \R^{\dim U}$ and 
$\widehat{T}_2(f) = f \times Jf \times Jf$ with the Jacobian matrix $Jf$ for a smooth map $f\colon U \to V$, 
we have  the functor 
$T_2\coloneqq\mathbb{L}\mathcal{Y}\widehat{T}_2 \colon  \mathsf{Diff} \to \mathsf{Diff}$ 
via the left Kan extension along the Yoneda functor $\mathcal{Y}$. 
Observe that as a diffeological space, 
\[
 T_2(X) = \colim_{P \in \D}(U_P\times \R^{\dim U_P}\times \R^{\dim U_P}).
\]
We may write $[x, v_1, v_2]_P$ for an element in $T_2(X)$ 
which has an element $(x, v_1, v_2)$ in $U_P\times \R^{\dim U_P}\times \R^{\dim U_P}$ for some plot $P$ as a representative.

\begin{defn}\label{defn:WR}
A map $g \colon T_2(X) \to \R$ is a \textit{weak Riemannian metric} on $X$ 
if the composite $g \circ \pi_{\widehat{T}_2(U_P)}\eqqcolon g(P)$ is a symmetric and positive covariant $2$-tensor on $U_P$ 
for each plot $P$ of $X$, where 
$\pi_{\widehat{T}_2(U_P)}$ denotes the canonical map 
$\widehat{T}_2(U_P)\to T_2(X)$. 
\end{defn}

\begin{rem}\label{rem:smoothness_g_and_g(P)}
In the definition above, each map $g \circ \pi_{\widehat{T}_2(U_P)}$ is smooth in the usual sense and then $g$ is smooth in the diffeological sense. 
In fact, $T_2(X)$ is endowed with the quotient diffeology; 
see Example \ref{ex:manifolds_subsets_products} (\ref{item:quotient_diff}). 
\end{rem}

We introduce the Riemannian metric on a diffeological space described in \cite[Page 3]{PIZ23}. 
To this end, we recall the definition of covariant tensors of a diffeological space.

Let $(X, \D^X)$ be a diffeological space and $\D^X(U)\subset\mathcal{D}^X$ the set of plots whose domains are the common domain $U$. 
Let $T^k(U)$ denote the set of covariant $k$-tensors on a domain $U$.
A \textit{covariant $k$-tensor} $\nu$ in the sense in Iglesias-Zemmour is a natural transformation which fits in the diagram 
\begin{equation}\label{eq:R-metric_g}
 \xymatrix@C60pt@R20pt{
 \mathsf{Euc}^{\text{op}} \rtwocell^{\D^X}_{T^k}{\hspace*{0.3cm} \nu} & 
 \mathsf{Sets}
 }. 
\end{equation}
Here, we regard $\D^X$ as a functor in which $\D^X(f) \colon D^X(V) \to \D^X(U)$ is defined by $\D^X(f)(P_V) = P_V\circ f$ for a smooth map $f\colon U \to V$.

We may write $\nu(P)$ for $\nu_{U_P}(P)$.
A covariant $k$-tensor $\nu$ is \textit{(anti-) symmetric} if $\nu(P)$ is (anti-) symmetric in the usual sense for each $P$.
A \textit{(differential) $k$-form} on a diffeological space $X$ is an anti-symmetric $k$-tensor on $X$.
The set of all $k$-forms on $X$ is denoted by $\Omega^k(X)$.

\begin{defn}[{\cite[p.~3]{PIZ23}}]\label{def:metric}
Let $X$ be a diffeological space.
A $2$-tensor $g$ on $X$ is a \emph{Riemannian metric} on $X$ if it satisfies the following three conditions.
\begin{enumerate}
\item\label{item:symmetry_IZ}
	(Symmetric)
	The tensor $g$ is symmetric.
\item\label{item:positivity_IZ}
	(Positivity)
	For all path $\gamma \in \text{Path}(X)\coloneqq C^{\infty}(\R,X)$, we have $g(\gamma) \ge 0$; that is,  
		$g(\gamma)_t(1, 1) \ge 0$ for $t \in \text{dom}(\gamma) =\R$, where $1 = (\frac{d}{dt})_t$ is the canonical base of 
		$T_t(\text{dom}(\gamma)) = \R$.
\item\label{item:definiteness_IZ}
	(Definiteness)
	If $g(\gamma)_t(1, 1) = 0$ for all $\gamma \in \text{Path}(X)$, 
	then we have $\alpha(\gamma)_t(1) = 0$ for any $\alpha \in \Omega^1 (X)$.
\end{enumerate}
\end{defn}

\begin{rem}
Another definition of the definiteness is proposed in \cite[p.~3]{PIZ23} with ``pointed differential forms''.
\end{rem}

\begin{prop}\label{prop:one-to-one} Let $(X,\D)$ be a diffeological space. 
There exists a one-to-one correspondence between the set of weak Riemannian metrics $g \colon T_2(X) \to \R$ and that of symmetric and positive covariant $2$-tensors in the sense of Definition~\ref{def:metric}. 
\end{prop}

\begin{proof}
A natural transformation $\{g(P)\}_{P\in \D}$ consisting of covariant $2$-tensors gives rise to the smooth map $g$ (see Remark~\ref{rem:smoothness_g_and_g(P)}). 
Given a map $g \colon T_2(X) \to \R$ such that $g(P)= g\circ \pi_{\widehat{T}_2(U_P)}$ is a $2$-tensor on $U_P$ for each $P\in\D$, 
we obtain a natural transformation $\{g(P)\}_{P\in \D}$. 
Then $g$ and $g(P)$ fit into the commutative diagram 
\begin{equation}\label{eq:g(P)}
\begin{split}
 \xymatrix@C30pt@R20pt{
 U_P\times \R^{\dim U_P}\times \R^{\dim U_P}\ar[d]_{\pi_{\widehat{T}_2(U_P)}} \ar[rd]^{g(P)} \\
 T_2(X) \ar[r]_g & \R. 
 }
\end{split}
\end{equation}
It is easy to see that $g$ is symmetric if and only if $g(P)$ is symmetric for each $P$.

We consider the equivalence between the positivity of $g(P)$ and the condition (\ref{item:positivity_IZ}) in Definition \ref{def:metric}. 
Suppose that $g(P)$ is positive for each plot $P$. 
Since a path $\gamma$ is a plot, it is immediate that the condition (\ref{item:positivity_IZ}) holds.

Conversely, suppose that $g(P)_r(v, v)=0$, where 
$P \in \D$, $r \in U_P$ and $v \in T_rU_P$. 
We apply the same argument as in the proof of \cite[243 Exercise 2]{PIZ25}. 
Let $\gamma_v(t) = r + tv$ and $\gamma^v=P\circ \gamma_v$. 
Then, we see that $g(P)_r(v, v)=g(\gamma^v)_0(1, 1)$. 
Thus, the condition (\ref{item:positivity_IZ}) implies that $g(P)$ is positive. 
As a consequence, the universality of the colimit yields the result. 
\end{proof}

In what follows, 
we call a diffeological space $X$ endowed with a weak Riemannian metric $g$ a \emph{weak Riemannian diffeological space} and denote it by $(X, g)$.

\begin{rem}
In order to define a weak Riemannian metric on a diffeological space $X$, the elastic conditions on $X$ as in \cite[Section 1.2]{B23} 
and \cite[Section 2.3.3]{B24} are not assumed. 
In particular, we do not require the condition that the tangent functor $T$ is compatible with limits; see \cite[Axiom (E1)]{B23}.

We observe that $T_2(X)$ is not in general diffeomorphic to the pullback $T(X)\times_X T(X)$ along the natural map $T(X) \to X$; 
see \cite[Examples 2.3.12 and 2.3.13]{B24} and \cite[Example 3.9]{M}
\end{rem}

\begin{defn}\label{defn:definiteness}
A weak Riemannian metric $g \colon T_2(X) \to \R$ is \textit{definite} 
if there exists a generating family ${\mathcal G}$ of the diffeology $\D$ of $X$ such that the symmetric positive covariant $2$-tensor $g(P)$, 
which corresponds to $g$ via the bijection in Proposition \ref{prop:one-to-one}, 
is definite in the usual sense for every $P \in \mathcal{G}$.  
\end{defn}

\begin{rem} 
Let $(X, g)$ be a weak Riemannian diffeological space with the diffeology $\D$. 
In general, if $\D$ is generated by the empty set, then the metric $g$ is definite. 
In this case, the diffeology $\D$ is indeed discrete; see \cite[1.67]{PIZ12}.   
\end{rem}

\begin{ex}\label{ex:manifold_G}
Let $N$ be a Riemannian manifold endowed with a metric $g_N$. 
We see that the metric $g_N$ is a definite weak Riemannian metric in the sense of Definition \ref{defn:definiteness} with respect to both generating families $\mathcal{G}_{\text{atlas}}$ and $\mathcal{G}_{\text{imm}}$ in Example~\ref{ex:manifolds_subsets_products} (\ref{item:generating_family}). 
\end{ex}

\begin{prop}\label{prop:IZ-Ours}
The definiteness of a weak Riemannian metric in Definition \ref{defn:definiteness} induces that in Definition \ref{def:metric}. 
\end{prop}

\begin{proof}
Let $g$ be a definite weak Riemannian metric. 
To show that $g$ is definite in the sense of Definition \ref{def:metric}, suppose that 
$g(\gamma)_t(1, 1) = 0$ for $\gamma \in \text{Path}(X)$. 
By assumption, 
there exist an open neighborhood $V_t$ of $t$ in $\mathbb{R}$, 
a plot $Q \in\mathcal{G}$ and a smooth map $f\colon V_t \to U_Q$ such that $\gamma|_{V_t} = Q \circ f$.
Then, we see that 
\[
 0=g(\gamma|_{V_t})_t(1, 1) 
 = g(Q \circ f)_t(1, 1) 
 = g(Q)_{f(t)}\left(\frac{df}{dt}(t), \frac{df}{dt}(t)\right).
\]
The definiteness of $g(Q)$ yields that $\dfrac{df}{dt}(t) = 0$.
Thus, for any $\alpha\in\Omega^1(X)$, we have
\begin{equation*}
 \alpha(\gamma)_t(1)
 =\alpha(\gamma|_{V_t})_t(1)
 = \alpha(Q)_{f(t)}\left(\frac{df}{dt}(t)\right)= \alpha(Q)_{f(t)}(0) = 0.   
\end{equation*}
This completes the proof. 
\end{proof}

\subsection{Manifolds in Riemannian diffeology}\label{subsection:IDMfd}
Let $C^\infty\mathsf{Mfd}$ be the category of (possibly infinite dimensional) manifolds modeled on locally convex spaces with convenient calculus; see \cite{K-M}. 
Observe that, by definition, 
a map $\R \to E$ to a locally convex vector space is a \textit{smooth curve} if all derivatives exist and are continuous. 
We call the final topology for the set of smooth curves into $E$ the \textit{$c^\infty$-topology} of $E$. 
A map $f \colon U \to V$ between $c^\infty$-open subsets of locally convex vector spaces is \textit{smooth} 
if $f\circ c$ is a smooth curve in $V$ for each smooth curve $c$ in $U$.

We recall the fully faithful functor 
$I\colon C^\infty\mathsf{Mfd} \to \mathsf{Diff}$ 
due to Kihara \cite[Lemma 2.5]{Kihara}. 
Here the underlying set of $I(M)$ is $M$ and the diffeology consists of all smooth maps into $M$ from open subsets of Euclidean spaces; 
that is, the diffeology is the standard one.

\begin{defn}(\cite[4.1 Definition]{Sc}) \label{defn:metric_on_Mfd}
Let $M$ be a manifold modeled on locally convex spaces. 
A \textit{weak Riemannian metric} $g$ is a smooth map 
\[
 g \colon TM\oplus TM \to \R, \ (v_x, w_x)\mapsto g_x(v_x, w_x)
\]
(where $TM\oplus TM$ is the Whitney sum) satisfying 
\begin{enumerate}[label=(\alph*)]
\item
	$g_x\coloneqq g|_{T_xM\times T_xM}$ is symmetric and bilinear for all $x \in M$, 
\item
	$g_x(v, v)\geq 0$ for all $v \in T_xM$, with $g_x(v, v) =0$ if and only if $v = 0$. 
\end{enumerate}
\end{defn}

\begin{prop}\label{prop:IDMfd}\label{prop:ID_Mfd}
A weak Riemannian metric on a manifold $M$ in $C^\infty\mathsf{Mfd}$ gives rise to a weak Riemannian metric on the diffeological space $I(M)$. 
\end{prop}

\begin{rem} 
Bastiani calculus is mainly used in \cite{Sc}; see \cite[Section 1.3]{Sc}.
We see that Bastiani calculus gives rise to convenient calculus; see \cite[A.7]{Sc} and \cite[Appendix B]{Kihara} for more details. 
Then, the definition of a metric in \cite[4.1 Definition]{Sc} is valid for infinite-dimensional manifolds under convenient calculus.

We may choose the category $C^\infty\mathsf{Mfd}_\text{conv}$ of manifolds modeled on locally convex spaces with Bastiani calculus. 
However the natural faithful functor $J\colon C^\infty\mathsf{Mfd}_\text{conv}\to \mathsf{Diff}$ is not full; see \cite[Corollary B.3]{Kihara}.
\end{rem}

\begin{proof}[Proof of Proposition \ref{prop:IDMfd}]
We consider the Whitney sum $TM\oplus TM$ of the tangent bundle $\rho \colon TM \to M$. 
Then, for each plot $P$ of $I(M)$, we have a commutative diagram of solid arrows
\begin{equation} 
 \xymatrix@C20pt@R15pt{
 U_P\times \R^{\dim U_p}\times \R^{\dim U_p}  \ar@/^15pt/[rrd]^-{\varphi_P}   \ar@/_15pt/[ddr]_-{\varphi_P'} \ar@{.>}[rd]^{(\varphi_P', \varphi_P)}& & \\
 & TM\oplus TM \ar[r] \ar[d] & TM\times TM \ar[d]^{\rho\times \rho} \\
 & M \ar[r]_{\Delta} & M\times M
}
\end{equation}
in which the smooth maps $\varphi_P'$ and  $\varphi_P$ are defined by $\varphi_P'(x, u, v) = P(x)$ and $\varphi_P(x, u, v)= (P_*(x, u), P_*(x, v))$, respectively. 
Then, there exists a smooth map $(\varphi'_P, \varphi_P)$ which makes the triangles commutative.

Moreover, we see that 
$(\varphi'_P, \varphi_P)= (\varphi'_Q, \varphi_Q)\circ \widehat{T}_2(f)$ for $f \colon U_P\to U_Q$ with $Q\circ f= P$.
Therefore, there exists a unique smooth map 
$\Theta_M \colon T_2(I(M)) \to I(TM\oplus TM)$ 
in $\mathsf{Diff}$ such that 
$I((\varphi'_P, \varphi_P)) = \Theta_M\circ \pi_{\widetilde{T}_2(U_P)}$ for each $P \in \D_{\text{std}}^M$.  
As a consequence, 
a weak Riemannian metric $g \colon TM\oplus TM\to \R$ gives rise to a weak Riemannian metric $I(g)\circ \Theta_M \colon T_2(I(M)) \to \R$ 
on the diffeological space  $I(M)$. 
\end{proof}

\begin{rem}\label{rem:InfiniteMfd_definiteness}
Let $M$ be an object in $C^\infty\mathsf{Mfd}$ which admits a weak Riemannian metric $g$. 
Then the metric is definite; see Definition  \ref{defn:metric_on_Mfd} (b).
Let $\mathcal{G}$ be the subset of the diffeology of $I(M)$ 
consisting of smooth maps $P \colon U_P\to M$ each of which the map $(P_*)_r \colon T_rU_P \to T_{P(r)}M$ induced by $P$ is injective for $r\in U_P$; 
see \cite[1.53 Definition]{Sc}. 
Then, we see that 
$(M, \langle \mathcal{G} \rangle, I(g)\circ \Theta_M)$ is a weak Riemannian diffeological space 
whose metric is definite with respect to $\mathcal{G}$; see Section \ref{sect:I(M)}. 
\end{rem}

\subsection{A category of  weak Riemannian diffeological spaces} 
We define a category of weak Riemannian diffeological spaces introducing the pullback of a metric.

\begin{defn}\label{defn:Pullback} 
Let $X$ be a diffeological space and $(Y, g)$ be a weak Riemannian diffeological space. 
Let $\varphi \colon X \to Y$ be a smooth map. 
The map $\varphi^*g$ defined by the composite
\[ 
 \xymatrix@C20pt@R15pt{
 T_2(X) \ar[r]^-{T_2(\varphi)} & T_2(Y) \ar[r]^-{g} & \R 
 } 
\]
is called the \textit{pullback of $g$} by $\varphi$. 
\end{defn}

\begin{rem}
Under the same notation as in Definition \ref{defn:Pullback},
consider the pullback $\varphi^*g$ of $g$. For any $[P, u, v, w ] \in T_2(X)$, we have
\[
 \varphi^*g(P)_u(v, w) = g(\varphi \circ P)_u(v, w).
\]
Since $g$ is a weak metric on $Y$ and $\varphi \circ P$ is a plot of $Y$, 
it follows that for each $P \in \D^X$, the pullback $\varphi^*g(P)$ is a positive, symmetric, covariant 2-tensor. 
Therefore, $\varphi^*g$ is a weak Riemannian metric on $X$.
\end{rem}

\begin{prop}\label{prop:An_induction}
Let $(X, g)$ be a weak Riemannian diffeological space whose metric $g$ is definite with respect to a generating family $\mathcal{G}$ and $i\colon A \to X$ an induction. 
Then, the pullback $i^*g$ is a weak Riemannian metric on $A$ 
which is definite with respect to the generating family 
\begin{equation}\label{eq:pullback_generating_family}
 i^*\mathcal{G}\coloneqq\{ P \in \D^A \mid i\circ P \in \mathcal{G}\}. 
\end{equation}
\end{prop}

\begin{proof}
We show that $\langle i^*\mathcal{G} \rangle$ coincides with $\D^A$. 
Since $\langle i^*\mathcal{G} \rangle \subset \D^A$ is obvious, 
it suffices to show that $\langle i^*\mathcal{G} \rangle \supset \D^A$.
For any $P \in \D^A$, since $i$ is an induction, we have $i \circ P \in \D^X$. 
Since $\mathcal{G}$ generates $\mathcal{D}^X$, 
it follows that for any $r \in U_P$, 
there exist an open neighborhood $V \subset U_P$ of $r$, 
an element $Q \colon U_Q \to X$ of $\mathcal{G}$, and a smooth map $f \colon V \to U_Q$ with $\circ P|_V = Q \circ f$.
Since $i$ is an induction, we have $(A, \D^A) \cong (\im(i), \D^{\im(i)}_{\text{sub}})$; 
see Proposition \ref{prop:induction}. 
Let $\bar{i} \colon \im(i) \to A$ be the inverse to $i$.
The fact that $i \circ \bar{i} \circ Q = Q \in \mathcal{G}$ enables us to deduce that $\bar{i} \circ Q\in i^*\mathcal{G}$.
Since the image of $Q\circ f=i \circ P|_V$ lies in $\im(i)$, 
we have
\[
 P|_V = \bar{i} \circ i \circ P|_V = \bar{i} \circ Q \circ f.
\]
This implies $P\in\langle i^*\mathcal{G}\rangle$.

The map $T_2(i)$ is smooth. 
Therefore, we see that  $g_A=g\circ T_2(i)$ is smooth. 
Thus, the symmetry and positivity of $g$ imply that  $g_A$ is a weak Riemannian metric.

To show that $g_A$ is definite with respect to $i^*\mathcal{G}$,
suppose $g_A(P)_r(v,v)=0$ for $P\in i^*\mathcal{G}$, $r\in U_P$ and $v\in T_rU_P$.
By definition we have
\[
 0=g_A(P)_r(v,v)=g(i\circ P)_r(v,v). 
\]
Since $i\circ P$ is in $\mathcal{G}$, it follows from the definiteness of $g$ that $v=0$.  
\end{proof}

\begin{rem}
We observe that the pullback \eqref{eq:pullback_generating_family} of $\mathcal{G}$ is smaller than that in the sense in \cite[1.75]{PIZ12}.
But \eqref{eq:pullback_generating_family} is enough to generate $\mathcal{D}^A$ when $i$ is an induction.
\end{rem}

\begin{defn}
Let $(X, g_X)$ and $(Y, g_Y)$ be weak Riemannian diffeological spaces, and  
$\varphi \colon X \to Y$ a smooth map. 
The map $\varphi$ is called an \emph{isometry} if $\varphi^*g_Y = g_X$.
\end{defn}

We denote by $\mathsf{RiemDiff}$ the category of weak Riemannian diffeological spaces and isometries.

\begin{prop} 
Let $\mathsf{Riem}C^\infty\mathsf{Mfd}$ be the category consisting of Riemannian manifolds modeled on locally convex spaces and isometries; 
see Section \ref{subsection:IDMfd}. 
Then $\mathsf{Riem}C^\infty\mathsf{Mfd}$ embeds in $\mathsf{RiemDiff}$.
\end{prop}

\begin{proof}
Let $(M, g)$ be a Riemannian manifold modeled on locally convex spaces.
By Proposition \ref{prop:IDMfd}, 
we regard $M$ as a weak Riemannian diffeological space. 
We then define a functor $R \colon \mathsf{Riem}C^\infty\mathsf{Mfd} \to \mathsf{RiemDiff}$ 
by sending a morphism $\varphi \colon (M, g_M) \to (N, g_N)$ to
$\varphi \colon (M, \D^M_{\mathrm{std}}, g_M) \to (N, \D^N_{\mathrm{std}}, g_N)$.
It is immediate that the functor gives rise to an injective map 
\[
 R_{M, N} \colon \Iso_{\mathsf{ \mathsf{Riem}C^\infty\mathsf{Mfd}}}(M, N) \to \Iso_{\mathsf{RiemDiff}}(R(M), R(N))
\]
for any objects $M$ and $N$ in $\mathsf{Riem}C^\infty\mathsf{Mfd}$.  
We show the surjectivity of the map $R_{M, N}$. 
Recall the fully faithful functor $I \colon C^\infty\mathsf{Mfd} \to \mathsf{Diff}$ and 
the smooth map $\Theta_{()}$ defined in the proof of Proposition \ref{prop:IDMfd}; 
see Section \ref{subsection:IDMfd}. 
For an isometry $\varphi \colon R(M) \to R(N)$, 
the naturality of the map $\Theta_{()}$ gives the diagram 
\begin{equation} 
 \xymatrix@C35pt@R20pt{
 T_2(M) \ar[r]^-{\Theta_M} \ar[d]_-{T_2(\varphi)}& TM\oplus TM \ar[r]^-{g_M} \ar[d]_-{T(\varphi)\oplus T(\varphi)} & \R \\
 T_2(N) \ar[r]_-{\Theta_N} & TN\oplus TN \ar[ru]_-{g_N} 
 }
\end{equation}
in which the trapezoid and the left square are commutative. 
We observe that $\varphi$ in $T(\varphi)$ is a smooth map from $M$ to $N$ in $C^\infty\mathsf{Mfd}$ with $I(\varphi)= \varphi$. 

Thus, in order to prove the fullness of the functor $R$, 
we show that the right triangle is commutative. 
To this end, it suffices to prove that $\Theta_M$ is surjective.

Let $\phi \colon U_x \to E_\phi$ be a chart around an element $x \in M$ with $\phi(x)=0$.
For an element $(x, u, v) \in T_xM\times T_xM \subset TM\oplus TM$, we define a plot $P \colon  \R^2  \to M$ by 
$P(s, t) = \phi^{-1}(su + tv)$.  
Then, we see that $(\Theta_M \circ  \pi_{\widetilde{T}_2(U_P)}) (0, \partial/\partial s, \partial/\partial t, ) = (x, u, v)$. 
This completes the proof. 
\end{proof}

\subsection{The Riemannian pseudo-distance}
Let $g \colon T_2(X) \to \R$ be a weak Riemannian metric on a diffeological space $X$. 
Then, we have a pseudo-distance $d$ on $X$ associated with the metric $g$ by applying the usual procedure, 
as in the case of a Riemannian manifold \cite{Mu}. 
In fact, 
the pseudo-distance $d \colon X\times X \to \R_{\geq 0}\cup\{\infty \}$ is defined by 
\begin{equation}\label{eq:d}
 d(x, y) = {\displaystyle\inf_{\gamma \in \text{Path}(X;x, y)} \ell(\gamma)},
 \quad  \text{where} \quad   
 \ell(\gamma) =\int_0^1 (g(\gamma)_t(1,1))^{\frac{1}{2}}dt
\end{equation}
and $d(x, y) =\infty$ if there is no smooth path connecting $x$ and $y$. 
Here  $\text{Path}(X;x, y)$ denotes the subset of $\mathrm{Path}(X)$ consisting of $\gamma$ with $\gamma(0) = x$ and $\gamma(1) =y$.

\begin{thm}\label{thm:D-top}
Let $d \colon X \times X \to \R_{\geq 0}$ be the pseudodistance on a connected diffeological space $X$ defined by a weak Riemannian metric 
$g \colon T_2(X) \to \R$. 
Then the $D$-topology of $X$ is finer than the topology $\mathcal{O}_d$ defined by $d$; 
that is, the D-topology contains $\mathcal{O}_d$. 
In particular, the function $d$ on $D(X)\times D(X)$ is continuous. 
\end{thm}

In order to prove the theorem, we need a lemma.

\begin{lem}\label{lem:d_E} 
Let $d_E$ be the Euclidean metric on $U_P$ and $x \in U_P\cap V$ for some domain $V \subset \R^{\dim U_P}$. 
Then there exists a real number $k >0$ and an open ball $B$ of $U_P\cap V$ centered at $x$ such that 
$\overline{B} \subset U_P\cap V$ and $d_P \leq kd_E$ on $B\times B$, 
where $d_P$ is the distance defined by the symmetric, positive covariant $2$-tensor $g(P)$.
Moreover, 
if the $2$-tensor $g(P)$ is definite, 
then, 
one has $\frac{1}{k} d_E \leq d_P \leq kd_E$ on $B\times B$ for some open ball $B$ of $U_P$ with 
$x\in B \subset \overline{B} \subset U_P\cap V$ and $k >0$. 
\end{lem}

\begin{proof}
This follows from the proof of \cite[Theorem 4.1.8]{Mu}. 
The upper bound is given by the continuity of $g(P)$. 
The definiteness of $g(P)$ yields the lower bound. 
\end{proof}

\begin{proof}[Proof of Theorem \ref{thm:D-top}] 
Let $\pi \colon  \coprod_{P\in \D}U_P \to X$ be the subduction. 
The argument in Remark \ref{rem:The_quotient_map} yields that 
$\pi \colon  \coprod_{P\in \D}U_P \to D(X)$ is the quotient map. 
We show that, for $V\in\mathcal{O}_d$, 
the set $\pi^{-1}(V)\cap U_P$ is open in $U_P$ with the usual topology for each $P \in \D$. 
Let $x$ be in $\pi^{-1}(V)\cap U_P$. 
Since $V\in\mathcal{O}_d$, we have a $\delta$-neighborhood $U_d(P(x); \delta)$ of $P(x)$ which is contained in $V$.  
Lemma \ref{lem:d_E} yields that 
$d_P \leq kd_E$ on $B\times B$ 
for some open ball $B$ of $U_P$ centered at $x$ with $\overline{B} \subset U_P$ and $k >0$.

Let $\e$ be a positive number less than $\min\{\delta, \text{the radius of} \ B\}$. 
Then, we see that $U_{d_E}(x; \frac{\e}{k}) \subset U_{d_P}(x; \e)\subset B$. 
The following Claim~\ref{claim:1} allows one to conclude that 
\[
 \pi(U_{d_E}(x; \frac{\e}{k})) \subset \pi(U_{d_P}(x; \e)) \subset  U_d(P(x); \e) \subset U_d(P(x); \delta) \subset V. 
\]
This completes the proof. 
\end{proof}

\begin{claim}\label{claim:1}
$\pi( U_{d_P}(x; \e)) \subset  U_d(P(x); \e)$.
\end{claim}

\begin{proof}
Let $q$ be in $U_{d_P}(x; \e)$. 
By definition, we see that on $U_P$, 
\[
 d_P(x, q) = \displaystyle\inf_{\gamma \in \text{Path}(U_P;x, q)}\ell_P(\gamma),
 \quad \text{where} \quad
 \ell_P(\gamma)=\int_0^1 (g(P)_{\gamma(t)}(\gamma', \gamma'))^{\frac{1}{2}}dt.
\] 
Since $d_P(x, q) < \e$, 
it follows that  there exists a path $\gamma\in\text{Path}(U_P;x, q)$ such that 
$d_P(x, q)  \leq \ell_P(\gamma) \eqqcolon \widetilde{\delta}  < \e$. 
Define a path $\widetilde{\gamma}$ by $\widetilde{\gamma}\coloneqq P\circ \gamma= \pi\circ \gamma$. 
Then, we see that
$g(\widetilde{\gamma})_t(1,1) = g(P\circ \gamma)_t(1,1)= \gamma^*(g(P))_t(1,1) = g(P)_{\gamma(t)}(\gamma',\gamma')$. 
This implies that 
$\ell(\widetilde{\gamma}) = \widetilde{\delta}$ and then 
$d(P(x), P(q)) \leq \ell(\widetilde{\gamma}) = \widetilde{\delta} < \e$;
see (\ref{eq:d}). 
It turns out that $\pi(q)$ is in $U_d(P(x); \e)$.
\end{proof}

\begin{rem}\label{rem:D-top}
In the proof of Theorem \ref{thm:D-top},
we do \textit{not} need the definiteness of the weak Riemannian metric $g$.  
We do not clarify when the topology $\mathcal{O}_d$ induced by a weak Riemannian metric $g$ 
on a diffeological space $X$ contains the $D$-topology on $X$. 
Suppose that the metric $g$ is definite and a generating family $\mathcal{G}$ gives the definiteness. 
Let $\mathcal{O}_{d_P}$ be the topology of $U_P$ defined by the metric $g(P)$.  
Observe that $\mathcal{O}_{d_P}$ is the usual Euclidean topology.
Then, the latter half of Lemma \ref{lem:d_E} yields that $P^{-1}(O)\cap U_P\in\mathcal{O}_{d_P}$ 
for every $D$-open subset $O\subset X$ and every plot $P\in\mathcal{G}$.
\end{rem}

\begin{defn}\label{defn:points} 
Let $X$ be a diffeological space. 
A generating family $\mathcal{G}$ of a diffeology of $X$ \textit{separates points} if 
for distinct points $p$ and $q$ in $X$,  
there exist a plot $P \in  \mathcal{G}$ and an open ball $B$ with center $x$ such that $P(x) = p$, 
$B \subset \overline{B} \subset U_P$ and each smooth path $\gamma$ from $p$ to $q$ admits a local lift $\widetilde{\gamma}$ 
in $U_P$ with $\widetilde{\gamma}(0) =x$ and $\text{Im} \ \! \widetilde{\gamma} \not\subset B$. 
\end{defn}

\begin{rem}\label{rem:a_manifold}
Both generating families $\mathcal{G}_{\text{atlas}}$ and $\mathcal{G}_{\text{imm}}$ of a manifold $M$ separate points. 
\end{rem}

With the technical condition above, we have the following result.

\begin{thm}\label{thm:condition_for_g_to_be_definite}
The pseudodistance $d$ defined by a weak Riemannian metric $g$ is indeed a distance 
provided $g$ is definite with respect to  a generating family $\mathcal{G}$ which separates points. 
\end{thm}

\begin{proof}
The proof is verbatim the same as that of \cite[Theorem 4.1.6]{Mu} by replacing the use of a chart with that of a plot.  
Suppose that 
$d(p, q) = 0$ for distinct points $p$ and $q$. We choose a plot $P \colon U_P \to X$ in $\mathcal{G}$ 
and an open ball $B$ with center $x$ which satisfy the condition in Definition \ref{defn:points}. 
The proof of \cite[Theorem 4.1.6]{Mu} allows us to deduce that there exists a positive number $k$ such that 
\begin{equation}\label{eq:vectors}
 \frac{1}{k} ||v|| \leq (g(P)_a(v, v))^{\frac{1}{2}} \leq k||v||
\end{equation}
for $(a, v) \in \overline{B}\times \R^{\dim U_P}$. 
Let $\gamma$ be a smooth path from $p$ to $q$ and 
$\widetilde{\gamma}$ a local lift in $U_P$ stating from a point $x$ with $P(x) =p$. 
Then there exists the smallest number $s\in(0,1]$ satisfying $u\coloneqq\gamma(s)\in\gamma([0,1])\cap\partial B$.
Let $\gamma_1\coloneqq\gamma|_{[0,s]}$.
Observe that $s>0$ and 
$\int_0^s (g(P)_{\widetilde{\gamma}_1(t)}(\widetilde{\gamma}_1',\widetilde{\gamma}_1'))^{\frac{1}{2}}dt \neq 0$. 
In fact, if the integration is equal to zero, then the definiteness condition allows us to conclude that 
$\widetilde{\gamma}_1$ is constant. 
However, $\widetilde{\gamma}_1(0) =x \neq u = \widetilde{\gamma}_1(s)$. 
With the radius $r$ of $B$, we have 
\[
 \ell(\gamma) = \int_0^1 (g(\gamma)_t(1, 1))^{\frac{1}{2}}dt
 \geq
 \int_0^s (g(P)_{\widetilde{\gamma}_1(t)}(\widetilde{\gamma}_1', \widetilde{\gamma}_1'))^{\frac{1}{2}}dt 
 \geq 
 \frac{1}{k}r.
\]
The last inequality follows from (\ref{eq:vectors}) and a change of variables in the integration. 
Then, we see that 
$d(p, q) = \inf \ell(\gamma) \geq \frac{1}{k}r$, 
which is a contradiction. 
\end{proof}

\begin{rem}
If $\pi\colon\coprod_{P\in\mathcal{G}}U_P\to X$ is a \emph{local subduction} 
(Example~\ref{ex:manifolds_subsets_products} (\ref{item:final_diffeology}), see also \cite[2.16]{PIZ12}), 
then we may replace the condition ``a generating family $\mathcal{G}$ separates points'' with a simpler one;
``for any distinct $p,q\in X$, there exists a plot $P\in \mathcal{G}$ and an open ball $B\subset U_P$ with center $x$ 
such that $P(x)=p$, $B\subset\overline{B}\subset U_P$ and $q\not\in P(B)$.''
Indeed, for any distinct $p$ and $q$, the following two conditions are equivalent;
\begin{enumerate}[label=(\arabic*),leftmargin=*]
\item
	for any $P\in\mathcal{G}$ with $p\in P(U_P)$ and any open ball $B$ centered at $x$ ($P(x)=p$), we have $q\in P(B)$
\item
	$d(p,q)=0$
\end{enumerate}
It is easy to see that (1) implies (2).
Suppose (1) does not hold and let $P\in\mathcal{G}$ be a plot that does not satisfy (1).
For any smooth path $\gamma$ from $p$ to $q$, 
we can find a local lift $\widetilde{\gamma}$ of $\gamma$ on $U_P$ defined near $0$.
Then, we can prove that $\ell(\gamma)>r/k$ ($r$ is the radius of $B$, $q\not\in P(B)$) 
as in the proof of Theorem~\ref{thm:condition_for_g_to_be_definite}.
\end{rem}

\section{A diffeological adjunction space}\label{sect:Adjunction_spaces}
We introduce an appropriate setting to construct a weak Riemannian diffeological space by attaching two such spaces.  
As seen in Example \ref{ex:theFirstExample}, the construction of the metric is applicable to the spaces of smooth maps.

We consider the diffeological adjunction space 
obtained by two inductions $i$ and $j$ into weak Riemannian diffeological spaces 
\begin{equation}\label{eq:embeddings}
 \xymatrix@C20pt@R15pt{
 X & A \ar[l]_-{i} \ar[r]^-{j}& Y. 
 }
\end{equation}
We assume further  that 
\begin{itemize}
\item[(I)]
	weak Riemannian metrics $g_X \colon T_2(X) \to \R$ and $g_Y \colon T_2(Y)\to \R$ on $X$ and $Y$ satisfy the condition that 
	$i^*g_X = j^*g_Y \colon T_2(A) \to \R$.
\end{itemize}
Then, we will construct a weak Riemannian metric 
$g \colon T_2(X\coprod_A Y) \to \R$, 
where $Z\coloneqq X\coprod_A Y$ is endowed with the quotient diffeology $\D^Z$.

Let $f \colon P \to Q$ be a morphism in $\D^Z$. 
Then, for $s \in U_P$,  we have a commutative diagram
\begin{equation}\label{eq:pushouts}
 \xymatrix@C20pt@R10pt{
 W' \ar[rr]^{f|_{W'}} \ar@{^{(}->}[d] \ar@/_40pt/[dddr]_(0.3){P_{XY}} & & W \ar@{_{(}->}[d] \ar@/^40pt/[dddl]^(0.3){Q_{XY}}\\
 U_P \ar[rr]^f \ar[rd]_P & &  U_Q \ar[dl]^Q \\ 
 & Z & \\
 & X\coprod Y \ar[u]_\pi&
}
\end{equation}
here open neighborhoods $W'$ and $W$ of $s$ and $f(s) \eqqcolon r$ respectively are taken 
so that $P_{XY}$ and $Q_{XY}$ are plots on $X$ \emph{or} $Y$. 
In fact, 
plots $P$ and $Q$ are in the quotient diffeology $\D^Z$ and then the local lifting condition of the plots gives the diagram. 
We define $g_{U_P} \colon T_2(U_P) \to \R$ by
$g_{U_P}((s, v_1, v_2)) = g(P_{XY})((s, v_1, v_2))$; see Definition \ref{defn:WR} for the notation. 
Observe that 
$g_{U_P}((s, v_1, v_2))=g_X([s, v_1, v_2])$ if $P_{XY}(W') \subset X$; see the diagram (\ref{eq:g(P)}).

Suppose that $P_{XY}$ and $Q_{XY}$ are plots $P_X$ on $X$ and $Q_Y$ on $Y$, respectively. 
Then, we see that 
$\text{Im} \ \!P_X  \subset i(A)$ and $\text{Im} \ \!(Q_Y\circ f) \subset j(A)$.  
Observe that an induction gives a diffeomorphism between the domain and its image; see Proposition \ref{prop:induction}. 
Thus
$P_A\coloneqq i^{-1}\circ P_X$ and $\rho\coloneqq j^{-1}\circ Q_Y\circ f$ are defined and 
it follows that 
\begin{equation}\label{eq:P--g}
P_A=\rho. 
\end{equation}

For a plot $P \colon U_P \to Y$, we may write $P_* \colon T_2(U_P) \to T_2(Y)$ for the canonical inclusion
\begin{equation}
 \{P\}\times T_2(U_P) \to T_2(Y)=\colim_{P \in \D_Y}(U_P\times\R^{\dim U_P}\times\R^{\dim U_P})
\end{equation}
for a weak Riemannian diffeological space $Y$.

\begin{lem}\label{lem:well_definedness}
With the notation above, suppose that $f_* \colon T_2(U_P) \to T_2(U_{Q})$ assigns $(r, u_1, u_2)$ to $(s, v_1, v_2)$. 
Then, one has 
$g_{U_P}((s, v_1, v_2)) = g_{U_Q}((r, u_1, u_2))$.  
\end{lem}

\begin{proof}
By the equalities in  (\ref{eq:P--g}), we have the commutative diagram of diffeological spaces and smooth maps
\begin{equation}\label{eq:g_g}
 \xymatrix@C40pt@R20pt{
 T_2(W') \ar@/_10pt/[rd]_{((P_X)|_{W'})_*} \ar@/^15pt/[rrr]^(0.5){(f|_{W'})_*} \ar[r]_-{(\rho)_* = (P_A)_*} &T_2(A) \ar[r]^{j_*} \ar[d]_{i_*}& T_2(Y)\ar[d]^{g_Y} & T_2(W)\ar[l]^{((Q_Y)|_W)_*}\\
 & T_2(X) \ar[r]_{g_X}& \R.
}
\end{equation}
Thus, it follows that 
\begin{align*}
 & g_{U_P}((s, v_1, v_2)) = g_X([s, v_1, v_2]) = g_X(i_*((P_A)_*([s,v_1, v_2]))) \\
 ={}& g_X(i_*(\rho_*([s,v_1, v_2]))) = g_Y(j_*(\rho_*([s,v_1, v_2]))) \\
 ={}& g_Y(((Q_Y)|_W)_*((f|_{W'})_*([s,u_1, u_2]))) = g_{U_Q}((r, u_1, u_2)).  
\end{align*}
We have the result. 
\end{proof}

\begin{thm}\label{thm:adjunction_g}
Under the assumption {\em (I)} for inductions in the diagram (\ref{eq:embeddings}), the map 
$g \colon T_2(Z) \to \R$ defined by 
$g((s, v_1, v_2)) \coloneqq g_{U_{P_X}}((s, v_1, v_2))=g_X([s, v_1, v_2])$ 
for $(s, v_1, v_2) \in T_2(U_{P_X})$
is a well-defined weak Riemannian metric on the  diffeological  adjunction space $Z=X\coprod_AY$. 
Moreover, if $g_X$ and $g_Y$ are definite, then so is $g$. 
\end{thm}

\begin{proof}
Recall the diagram (\ref{eq:pushouts}). 
The proof of Lemma \ref{lem:well_definedness} is valid for the case where 
$P_{XY}$ and $Q_{XY}$ are both plots on $X$ or $Y$. 
In particular, we choose an inclusion $f \colon U_P \to U_Q$ in the diagram, where $s \in U_P$. 
Then, the commutativity of the diagram (\ref{eq:pushouts}) implies that 
the definition of $g_{U_P}$ does not depend on the choice of a neighborhood of $s$.
Therefore, the map $g$ is well defined. 
We show the smoothness of $g$. 
A plot $R\colon U_R \to T_2(Z)$ locally factors through $\{P\} \times T_2(U_P)$. 
Moreover, $R$ locally factors through an open subset $T_2(W')$ of $\{P\} \times T_2(U_P)$ 
which is used when defining the metric $g$. 
We see that $g$ is indeed $g_X$ or $g_Y$ on the open subset $T_2(W')$. 
This implies that $g$ is smooth.

The latter half of the assertion follows from the definition of $g$. 
In fact, let $\mathcal{G}_X$ and $\mathcal{G}_Y$ be generating families 
which give the definiteness of $g_X$ and $g_Y$, respectively. 
Then, the definiteness of $g$ is given with the set $(\rho_X)_*\mathcal{G}_X \cup (\rho_Y)_*\mathcal{G}_Y$, where, 
$\rho_X \colon X \to Z$ and $\rho_Y \colon Y \to Z$ are the natural maps.
\end{proof}

A smooth embedding between manifolds is an induction. 
Thus, Theorem \ref{thm:adjunction_g} provides a crucial example. 
An adjunction diffeological space of $M$ and $N$ does not necessarily have a generating family 
which separates points even if $M$ and $N$ do.

\begin{ex}\label{ex:Y}
We consider the pushout diffeological space $\mathsf{Y} \coloneqq \R_1\coprod_{(1, \infty)}\R_2$ 
of the induction $i\colon (1, \infty ) \to \R$ along itself, where $\R_k$  ($k=1,2$) denotes the copy of $\R$. 
By virtue of Theorem \ref{thm:adjunction_g}, 
we see that the usual metric on $\R$ gives rise to a weak Riemannian metric on $\mathsf{Y}$. 
A point $x \in \R_k$ is denoted by $x_k$ for $k=1,2$. 
The subduction $\pi \colon \R_1 \coprod \R_2 \to \mathsf{Y}$ gives rise to the quotient map  
$D(\pi) \colon \R_1  \coprod \R_2 \to D(\mathsf{Y})$; 
see the argument in Remark \ref{rem:The_quotient_map}. 
Thus, we see that 
$D(\mathsf{Y})$ is non--Hausdorff. 
In fact, the distinct points $[1_1]$ and $[1_2]$ is not separated by any neighborhoods of the points.

For each $\e >0$, 
we have a smooth path $\gamma$ from $[1_1]$ to $[1_2]$ with $\ell(\gamma) < \e$. 
This yields that $d([1_1], [1_2]) = 0$. 
We see that the diffeological space $\mathsf{Y}$ does not satisfy the condition in Definition \ref{defn:points}.

Moreover, this example implies that the inverse of the inclusion relation of the topologies in Theorem \ref{thm:D-top} does not hold in general even if the metric $g$ is definite. 
To see this, we write $(a, b)_k$ for the open interval $(a, b)$ in $\R_k$. 
Then, the set $D(\pi)((1-\e ,1+\e)_1)$ is open in $D(\mathsf{Y})$ for a positive number $\e>0$. 
However, the set $\pi((1-\e ,1+\e)_1)$ is not in $\mathcal{O}_d$. 
In fact, 
$\pi((1-\e ,1+\e)_1)$ does not contain the $\delta$-open ball $U_d([1_1], \delta)$ with center $[1_1]$ 
for every $\delta$ with $0 < \delta < \e$. 
The element $[(1-\frac{\delta}{2})_2]$ is in $U_d([1_1], \delta)$ but not in $\pi((1-\e ,1+\e)_1)$. 
This follows from the fact that 
\[
 d([(1-\frac{\delta}{2})_2], [1_1]) \leq  d([(1-\frac{\delta}{2})_2], [1_2]) + d([1]_2, [1]_1)
 =
 d([(1-\frac{\delta}{2})_2], [1_2]) \leq \frac{\delta}{2}.
\]   
\end{ex}

We refer the reader to \cite[Section 6]{Hicks} for non-Hausdorff Riemannian manifolds; 
see also \cite{O} for many examples of non-Hausdorff manifolds.  
\begin{ex}
Theorem \ref{thm:adjunction_g} allows us to conclude that 
the adjunction space $M\coloneqq\R\coprod_{(1, \infty)}\R^2$ is a definite weak Riemannian diffeological space 
which is not a manifold. Observe that $D(M)$ is non-Hausdorff.
\end{ex}

\begin{ex}\label{ex:+}
We consider the pushout diffeological space $\boldsymbol{+} \coloneqq \R_1\coprod_{\{0\}}\R_2$. 
We see that the diffeology of $\boldsymbol{+}$ has a generating family which separates points; 
see Definition \ref{defn:points}. 
Moreover, the usual metrics on $\R=\R_1$ and $\R=\R_2$ satisfy the assumption (I). 
\end{ex}

\section{A diffeological mapping space}\label{section:MappingSpaces}
In this section, 
we focus on a weak Riemannian metrics on a mapping space and the pseudo-distance associated with the metric. 
The definiteness of the metric is also discussed introducing a subdiffeology of the functional diffeology.

\subsection{Weak Riemannian metrics on mapping spaces}\label{ss:metric_on_MappingSapces} 
In \cite[ Exercise (3)]{PIZ23}, a Riemannian metric on $C^\infty(\R, \R^3)$ is discussed. 
Riemannian metrics on diffeological mapping spaces including the metric due to Iglesias-Zemmour are interpreted in our framework.

Let $M$ be a closed orientable finite dimensional manifold and $N$ a connected
weak Riemannian diffeological space.
Let $\{ (V_\lambda, \varphi_\lambda)\}_{\lambda \in \Lambda}$ be an atlas of $M$. 
Then, for a plot $P\in \D_{\text{func}}$ of the diffeological mapping space $C^\infty(M, N)$, 
$r \in U_P$ and tangent vectors $v,w\in T_rU_P$, 
we define a smooth map $\Theta_{g_N}(P)(v, w)(r, \text{--}) \colon M \to \R$ by 
\[
 \Theta_{g_N}(P)(v, w)(r, m)
 = g_N(\ad(P)\circ (1\times \varphi_\lambda^{-1}))_{(r, \varphi_\lambda(m))}(\underline{v}, \underline{w})
\]
for $m \in V_\lambda$, where 
$\underline{u} =(u, 0)\in T_rU_P\times T_{\varphi_{\lambda}(m)}\varphi_{\lambda}(V_{\lambda})=T_rU_P\times\R^{\dim M}$. 
For the functional diffeology $\D_{\text{func}}$, 
we define a metric $g$ on $C^\infty(M, N)$ by 
\begin{equation}\label{eq:metric_function}
g(P)_r(v, w) = \int_M \Theta_{g_N}(P)(v, w)(r, m)\vol_M
\end{equation}
for $r \in U_P$, where $\vol_M$ is a fixed $(\dim M)$-form on $M$.

The following Lemma~\ref{lem:Theta_simplified} shows the well-definedness of $\Theta_{g_N}$.

\begin{lem}\label{lem:Theta_simplified}
$\Theta_{g_N}(P)(v,w)(r,m)=g_N(\ev_m\circ P)_r(v,w)$.
\end{lem}

\begin{proof}
Define $j_{m,\lambda}\colon U_P\to U_P\times\varphi_{\lambda}(V_{\lambda})$ by $j_{m,\lambda}(x)\coloneqq(x,\varphi_{\lambda}(m))$.
Then $\underline{v}=(j_{m,\lambda})_*(v)$ and 
\begin{align*}
 \Theta_{g_N}(P)(v,w)(r,m)
 &=g_N(\ad(P)\circ(\id_{U_P}\times\varphi_{\lambda}^{-1}))_{(r,\varphi(m))}(\underline{v},\underline{w})\\
 &=g_N(\ad(P)\circ(\id_{U_P}\times\varphi_{\lambda}^{-1})\circ j_{m,\lambda})_r(v,w).
\end{align*}
It is easy to see that $\ad(P)\circ(\id_{U_P}\times\varphi_{\lambda}^{-1})\circ j_{m,\lambda}=\ev_m\circ P$.
\end{proof}

To see the smoothness of $g$, let $V$ and $W$ be vector fields on $U_P$. 
Since $g_N$ is a weak Riemannian metric on $N$, $g_N(\ad(P) \circ (\id_{U_P} \times \varphi^{-1}_\lambda))$ is a 2-tensor field on $U_P \times U_\lambda$, 
where $(V_\lambda, \varphi_\lambda)$ is a coordinate neighborhood of $m \in M$, and the function on $U_P \times V_\lambda$ defined by 
\begin{equation*}
 (r, \varphi^{-1}(m)) \mapsto g_N(\ev_m\circ P)_r(V(r), W(r))
\end{equation*}
is smooth. 
Thus the map $r\mapsto g(P)_r(V(r), W(r))$ on $U_P$ is smooth and hence $g(P)$ is a smooth 2-tensor.

We show that $g$ is compatible with coordinate changes. Let $ F \colon V \to U_P $ be a $ C^\infty $-map.  
Then, it follows that 
\[
 g_N(\ev_m\circ(P\circ F))_r(v,w)
 =g_N(\ev_m\circ P)_{F(r)}(dF_r(v),dF_v(w))
\]  
Thus, we have   
\[
 \Theta_{g_N}(P\circ F)(v, w)(r, m) = \Theta_{g_N}(P)(dF_r(v), dF_r(w))(F(r), m).
\]  
Therefore, it follows that 
\[  
 g(P\circ F)_r(v, w)
 = g(P)_{F(r)}(dF_r(v), dF_r(w)) 
 = F^*g(P)_r(v, w).  
\]  
This implies that $g$ is a diffeological covariant 2-tensor field on $ C^\infty(M, N) $.  
Symmetry and positivity follow from the fact that $g_N$ is a metric.  
Moreover, we have

\begin{prop}\label{prop:distance_on_Func} 
For the weak Riemannian metric $g$ on $(C^\infty(M, N), \D_{\text{\em func}})$ in \text{\em (\ref{eq:metric_function})}, 
the pseudo-distance $d$ defined by $g$ is indeed a distance provided the pseudo-distance $d_N$ defined by $g_N$ is a distance on $N$. 
\end{prop}

\begin{proof}
Assume that 
$d_{C^\infty(M, N)}(f_0, f_1) =  \inf_{\gamma \in \mathrm{Path}(C^\infty(M, N); f_0, f_1)} \ell(\gamma) = 0$ for $f_0, f_1 \in C^\infty(M, N)$. 
Then, it follows form  the definition of the length $\ell$ that there exists a sequence of paths $\{\gamma_n\}_{n=1,2,\dots}$ from $f_0$ to $f_1$ such that
\begin{equation}\label{eq:converges_to_zero}
 0
 = \lim_{n\to \infty}  \int_0^1 (g(\gamma_n)_s(1,1))^{\frac{1}{2}} \, ds
 = \lim_{n\to \infty}  \int_0^1\left(\int_M  g_N(\ev_m\circ \gamma_{n})_s(1,1) \vol_M\right)^{\frac{1}{2}} \! ds.
\end{equation}
In what follows, we write $h_n(m, s)$ for $g_N(\ev_m\circ \gamma_{n})_s(1,1)$. 
By applying the Cauchy--Schwarz inequality, we have
\begin{align*}
 \int_M  \left( h_n(m, s) \right)^{\frac{1}{2}} \cdot 1 \ \vol_M
 \leq
 \left(\int_M  h_n(m, s) \ \vol_M\right)^{\frac{1}{2}} \cdot \left(\int_M  1 \ \vol_M\right)^{\frac{1}{2}}.  
\end{align*}
Thus,
\begin{align*}
 0 \leq \int_0^1 \left(\frac{1}{\vol(M)^{\frac{1}{2}}} \ \int_M  \left( h_n(m, s) \right)^{\frac{1}{2}} \vol_M \right) ds
 \leq 
 \int_0^1\left(\int_M  h_n(m, s) \  \vol_M\right)^{\frac{1}{2}} ds,
\end{align*}
where $\vol(M) = \int_M \vol_M$. 
Taking the limit as $n \to \infty$, we have
\begin{align*}
 \lim_{n\to\infty} \int_0^1 \left(\frac{1}{\vol(M)^{\frac{1}{2}}}\int_M  \left( h_n(m, s) \right)^{\frac{1}{2}} \vol_M \right) ds = 0
\end{align*}
by \eqref{eq:converges_to_zero}.
Since $\vol(M)$ is positive constant, it follows from the Fubini theorem that  
\begin{align*}
 0 &= \lim_{n\to\infty} \int_0^1 \left( \int_M  \left( h_n(m, s) \right)^{\frac{1}{2}} \vol_M \right) ds 
   = \lim_{n\to\infty} \int_M \left( \int_0^1  \left( h_n(m, s) \right)^{\frac{1}{2}} ds \right) \vol_M.
\end{align*}
By applying \cite[3.12 Theorem]{Ru}  to the integration on $M$, we have 
\begin{align}\label{eq:almostM_1} 
 \lim_{k\to \infty} \ell(\ev_m\circ\gamma_{n_k}) = \lim_{k\to \infty}  \int_0^1\left( h_{n_k}(m, s)\right)^{\frac{1}{2}} ds  =0 
\end{align}
for $m$ almost everywhere in $M$ and some subsequence $\{\gamma_{n_k}\}$ of $\{\gamma_{n}\}$. 
Let $F_M$ be the subset of $M$ consisting of $m\in M$ that does not satisfy \eqref{eq:almostM_1}.
Then $F_M$ is of measure zero. 
For each fixed element $m \in (F_M)^c$, the smooth map $\ev_m\circ\gamma_{n_l}$ is regarded as a path on $N$ from $f_0(m)$ to $f_1(m)$.
This yields that 
$d_N(f_0(m), f_1(m)) = 0$ for $m \in (F_M)^c$. 
Since $d_N$ is a distance, it follows that $f_0 = f_1$ on $(F_M)^c$.

We observe that the composite
\[
 \varphi\colon M
 \xrightarrow{f_0\times f_1} D(N\times N)\to D(N)\times D(N)
 \xrightarrow{d_N}\R_{\ge 0}
\]
is continuous.
This is because
\begin{enumerate}[label=(\alph*),leftmargin=*]
\item
	$f_0\times f_1$ is continuous with respect to the $D$-topology on $N\times N$ since $f_0$ and $f_1$ are smooth, and
\item
	$d_N\colon N\times N\to\R_{\ge 0}$ is continuous by Theorem \ref{thm:D-top}. 
\end{enumerate}
We see that $x \in F_M$ if $f_0(x) \neq f_1(x)$. 
Suppose that $f_0(x_0)\neq f_1(x_0)$ for some $x_0\in M$.
Then $\varphi(x_0)>0$ and $U\coloneqq\varphi^{-1}(\R_{>0})$ is an open neighborhood of $x_0$. 
This yields that $U \subset F_M$. However, since the measure of $F_M$ is zero,
$F_M$ cannot contain any open set, which is a contradiction. 
\end{proof}

\subsection{A subdiffeology of the functional diffeology.}\label{sect:D'} 
We consider the definiteness of a weak Riemannian metric on a diffeological mapping space.  
Let $X$ and $Y$ be diffeological spaces and $\mathcal{G}$ a generating family of $\D^Y$.  
Then, the set $C^\infty(X, Y)$ of smooth maps is endowed with the functional diffeology $\D_{\text{func}}$; 
see Example \ref{ex:manifolds_subsets_products} (\ref{item:functional_diffeology}).  
We introduce a subdiffeology of $\D_{\text{func}}$. 
To this end, we consider the following condition (E) for a plot $P\in \D_{\text{func}}$. 
\begin{itemize}
\item[(E)]
For $r \in U_P$ and $m \in X$, there exists an open neighborhood $W_{r, m}$ of $r$ in $U_P$ such that the composite 
$
\xymatrix@C20pt@R15pt{
W_{r, m}  \ar@{^{(}->}[r]  &  U_P \ar[r]^-P  & C^\infty(X, Y) \ar[r]^-{\ev_m} & Y\\ 
}
$
is in $\mathcal{G}$, where $\ev_m$ denotes the evaluation map at $m$. 
\end{itemize}
Let $\mathcal{F}^{XY}_{\mathcal{G}}$ be the subset of plots in $\D_{\text{func}}$ each of which satisfies the condition (E). 
We may write $\mathcal{F}_{\mathcal{G}}$ for $\mathcal{F}^{XY}_{\mathcal{G}}$ when the domain $X$ and the codomain $Y$ are clear in the context. 
Let $\D'\coloneqq\langle\mathcal{F}^{XY}_{\mathcal{G}}\rangle$ be the diffeology generated by $\mathcal{F}^{XY}_{\mathcal{G}}$.  
It is immediate that $\D'$ is contained in $\D_{\text{func}}$ and that 
the identity map $(C^\infty(X, Y), \D') \to (C^\infty(X, Y), \D_{\text{func}})$ is smooth.

\begin{rem}
Consider the case that $N$ is a Riemannian manifold.
Suppose $P\colon U_P\to C^{\infty}(M,N)$ is such that the map $\ad(P)(-,m)=\ev_m\circ P\colon U_P\hookrightarrow N$ is an embedding for any $m\in M$.
For example $P$ is such a plot if $U_P\cong\mathrm{Int}\, D^{\dim N}$ and $\ad(P)\colon U_P\times M\to N$ is a framed immersion.
Then clearly $P$ itself satisfies the condition (E), and $\mathcal{F_{\mathcal{G}}}$ is not empty if $\mathcal{G}$ contains such a plot.
\end{rem}

\begin{ex}\label{ex:(E)}
The diffeology $\D'$ does not coincide with $\D_{\text{func}}$ in general. 
In fact, we consider the case where $X=Y = U$ is an open subset of $\R^n$ with $n\geq 1$ and $\mathcal{G} \coloneqq \{\id \colon U \to U \}$. 
Then, the standard diffeology of $U$ is generated by $\mathcal{G}$. 
In this case, for any $P\in\mathcal{F}_{\mathcal{G}}$, the image of $P$ consists of constant maps; 
for any $r\in U_P$ and $m\in U$, there exists an open neighborhood $W=W_{r,m}$ of $r$ in $U_p$ such that $\ev_m\circ P|_W=\id_U$ (in particular $W=U_P=U$ must hold), 
and hence, 
for any $x\in W$, we have $P(x)(m)=x$.
The same holds for every $Q\in\mathcal{D}'=\langle\mathcal{F}_{\mathcal{G}}\rangle$ because $Q$ locally factors through some $P\in\mathcal{F}_{\mathcal{G}}$.
Therefore $\D'$ is strictly finer than $\D_{\text{func}}$.
\end{ex}

The following lemma describes fundamental properties of the diffeology $\mathcal{D}'$. 

\begin{lem}\label{lem:C}
Let $f\colon X\to Y$ be a smooh map.
\begin{enumerate}[wide,leftmargin=0pt,label=(\arabic*)]
\item
	Let $Z$ be a diffeological space with a generating family $\mathcal{G}_Z$ of the diffeology. Then, the map 
	\[
	 f^* \colon (C^\infty(Y, Z), \langle\mathcal{F}^{YZ}_{{\mathcal{G}_Z}}\rangle) \to (C^\infty(X, Z), \langle\mathcal{F}^{XZ}_{{\mathcal{G}_Z}}\rangle)
	\]
	defined by $f^*(\varphi) = \varphi\circ f$ is smooth.
\item
	Let $\mathcal{G}_Y$ be a generating family of the diffeology of $Y$ and $\mathcal{G}_X$ be the pullback $f^*\mathcal{G}_Y$ of $\mathcal{G}_Y$ by $f$; 
	see (\ref{eq:pullback_generating_family}).  
	Then, for any $Z$, the map 
	\[
	 f_* \colon (C^\infty(Z, (X,\langle\mathcal{G}_X\rangle)), \langle\mathcal{F}^{ZX}_{\mathcal{G}_X}\rangle)
	 \to (C^\infty(Z, Y), \langle\mathcal{F}^{ZY}_{\mathcal{G}_Y}\rangle)
	\]
	defined by $f_*(\psi)= f\circ \psi$ is smooth. 
\end{enumerate}
\end{lem}

\begin{proof}
Since both $f\colon X\to Y$ and $f\colon(X,\langle\mathcal{G}_X\rangle)\to Y$ are smooth, 
it follows that the maps $f^*$ and $f_*$ are smooth with respect to the functional diffeology. 
Moreover, we see that $\ev_x\circ(f^*\circ Q)= \ev_{f(x)}\circ Q$ for every $Q \colon U_Q\to C^{\infty}(Y,Z)$ and $x \in X$. 
Thus, if $Q\in\mathcal{F}_{\mathcal{G}_Z}^{YZ}$, then the map $\ev_x\circ(f^*\circ Q)$ restricted to an appropriate domain is in $\mathcal{G}_Z$. 
This implies that $f^*\circ Q\in\mathcal{F}_{\mathcal{G}_Z}^{XZ}$, proving (1). 
The fact that $\ev_z\circ f_*=f\circ\ev_z$ and the definition of $\mathcal{G}_X$ yield that 
$f_*\circ P$ is in $\mathcal{F}_{\mathcal{G}_Y}^{ZY}$ for each $P \in \mathcal{F}_{\mathcal{G}_X}^{ZX}$. 
We have the results. 
\end{proof}

\begin{lem}\label{lem:a_point} 
Let $\mathcal{G}$ be a generating family of the diffeology of a diffeological space $Y$.  
The functional diffeology $\D_{\text{\em func}}$ of $C^\infty(* , Y)$ coincides with $\langle \mathcal{F}_{\mathcal{G}}^{*Y} \rangle$.
\end{lem}

\begin{proof}
It suffices to show that $\D_{\text{func}} \subset  \langle \mathcal{F}_{\mathcal{G}}^{*Y} \rangle$.
For a plot $P\in\D_{\text{func}}$, the adjoint $\text{ad}(P) \colon U_P \times * \to Y$ is regarded as a plot of $Y$. 
Thus, for each $r \in U_P$, there exist a neighborhood $W_r$ of $r$ in $U_P$,  $Q \in \mathcal{G}$ and a smooth map $h \colon W_r \to U_Q$ such that 
$\text{ad}(P)|_{W_r\times *} = \widetilde{Q}\circ (h\times *)$, 
where $\widetilde{Q}$ is defined by $\widetilde{Q} \coloneqq Q\circ \text{pr} : U_Q\times * \to Y$ with the projection $\text{pr}$ to the first factor.  
We see that $P|_{W_r} = \text{ad}(\widetilde{Q})\circ h$ and $\ev_*\circ\ad(\widetilde{Q}) = Q$. 
This implies that $P$ is in $\langle \mathcal{F}_{\mathcal{G}}^{*Y} \rangle$. 
\end{proof}

Let $s \colon Y \to C^\infty(X, Y)$ be the section of $\ev_x$ at any $x$; that is, $s(y)(x) = y$ for $y \in Y$ and $x \in X$.
\begin{lem}\label{lem:an_induction} 
The section $s \colon (Y, \langle \mathcal{G} \rangle) \to (C^\infty(X, Y), \D')$ is an induction. 
\end{lem}

\begin{proof} 
The map $q \colon Y \to (C^\infty(*, Y), \D_\text{func})$ defined by $q(y)(*) = y$ is smooth. 
By Lemma \ref{lem:C}, we see that the trivial map $u \colon X\to *$ induces the smooth map 
$u^* \colon (C^\infty(*, Y), \langle\mathcal{F}^{*Y}_{\mathcal{G}}\rangle) \to (C^\infty(X, Y), \langle\mathcal{F}_{\mathcal{G}}\rangle)$. 
Since $s = u^*\circ q$, it follows from Lemma \ref{lem:a_point} that the section $s$ is smooth.

For a parametrization $Q \colon U \to Y$, suppose that $s\circ Q \in \D'=\langle \F_\mathcal{G} \rangle$. 
Since $\ev_x$ is smooth for each $x$, it follows that $Q =\ev_x\circ s\circ Q \in \langle \mathcal{G} \rangle$. We see that $s$ is an induction.  
\end{proof}

\begin{thm}\label{thm:D'}
Let $g$ be the weak Riemannian metric on $C^\infty(M, N)$ defined as above, where the diffeology is restricted to $\D'\coloneqq \langle \mathcal{F}_{\mathcal{G}}\rangle$. 
Then the metric $g$ is definite with respect to the  generating family $\mathcal{F}_{\mathcal{G}}$ in the sense of Definition \ref{defn:definiteness}. 
\end{thm}

\begin{proof}
The inclusion of diffeologies gives rise to a natural transformation $\iota\colon \D' \Longrightarrow \D_{\text{func}}$. 
By combining $\iota$ with the metric $g$ and applying Proposition \ref{prop:one-to-one}, 
we have a weak Riemannian metric on $C^\infty(M, N)$; see the diagram (\ref{eq:R-metric_g}).

We show the definiteness of the metric. 
For a plot $P$ in the generating family $\mathcal{F}_{\mathcal{G}}$ of $\D'$, 
assume that $g(P)_r(u, u) = 0$, where $u\in T_rU_p$. 
Since $g_N$ is positive, we see that $g_N(\ev_m\circ P)_r(u,u)=0$ for each $m \in M$. 
Since $P$ satisfies the condition (E), by restricting $U_P$ to $W_{r, m}$, the map $\ev_m\circ P$ is in $\mathcal{G}$. 
The definiteness of $g_N$ enables us to conclude that $u =0$. 
This completes the proof. 
\end{proof}

We recall the section $s\colon (N, \langle \mathcal{G}  \rangle) \to (C^\infty(M, N), \D')$ in Lemma \ref{lem:an_induction}.

\begin{prop}\label{prop:an_extension}
One has the commutative diagram
\[
 \xymatrix@C30pt@R20pt{
 T_2((C^\infty(M, N), \D')) \ar[r]^-g  & \R \\
 T_2((N, \langle \mathcal{G} \rangle)) \ar[u]^{s_*} \ar[ru]_(0.6){\ \ \ \ (\int_M\mathrm{vol}_M) \times g_N} & 
 }
\]
\end{prop}

\begin{proof}
This result follows from Lemma \ref{lem:Theta_simplified} and  the definition of the metric $g$; see (\ref{eq:metric_function}). 
\end{proof}

\begin{ex}\label{ex:mapping_spaces_pushout}
Let  $(N, \langle \mathcal{G} \rangle)$ be a  weak Riemannian diffeological space whose metric  is definite with respect to $\mathcal{G}$. 
For example, we can choose the adjunction space in Theorem \ref{thm:adjunction_g} as such a diffeological space. 
Let $M$ and $M'$ be closed orientable Riemannian manifolds with $\mathrm{vol}(M) =\mathrm{vol}(M')$. 
Then, 
Lemma \ref{lem:an_induction}, Proposition \ref{prop:an_extension} and Theorem \ref{thm:adjunction_g} allows us to obtain a definite weak Riemannian diffeological space of the form $C^\infty(M, N)\coprod_NC^\infty(M', N)$.  
\end{ex}

\subsection{An example: free loop spaces}\label{ss:free_loop_space}
For a weak Riemannian diffeological space $N$, let the free loop space $LN\coloneqq C^\infty(S^1, N)$ be endowed with the functional diffeology.

We regard $S^1=[0,2\pi]/(0\sim 2\pi)$ and define $S^1\vee S^1$ as the pushout of the diagram
\begin{equation}\label{eq:the_Pushout}
\begin{split}
 \xymatrix@C25pt@R15pt{ 
 \ast\ar[r]^-{*\mapsto 0}\ar[d]_-{*\mapsto 0} & S^1\ar[d]^-i \\
 S^1\ar[r]_-j & S^1\vee S^1.
 }
\end{split}
\end{equation}
Let $S^1_\star\subset S^1\vee S^1$ ($\star=\text{left},\text{right}$) be the image of the first / second copy of $S^1$ 
via the subduction $\pi \colon S^1\coprod S^1\to S^1\vee S^1$. 
The inclusions $i_\star \colon S^1_\star \to S^1\coprod  S^1$ give the smooth map $\eta \colon C^\infty(S^1\coprod S^1, N) \to LN\times LN$ defined by 
$\eta(f) = (f\circ i_{\text{left}}, f\circ i_{\text{right}})$. 
It is readily seen that $\eta$ is a diffeomorphism.

Moreover, 
the subduction $\pi$ gives rise to an induction $\pi^* \colon C^\infty(S^1\vee S^1, N) \to C^\infty(S^1\coprod S^1, N)$. 
This follows from the definition of functional diffeology and the fact that the product preserve a subduction. 
We will see in Section \ref{sect:warped_products} that a weak Riemannian metric $g\oplus g$ is induced on the product space $LN\times LN$, 
and by Proposition~\ref{prop:An_induction} the metric $g\oplus g$ restricts to that on $C^{\infty}(S^1\vee S^1,N)$. The restricted metric is denoted by $g_\vee$.

Fix a monotonically increasing smooth function $b\colon\R\to[0,2\pi]$ (in a weak sense) satisfying
$b(s)=0$ if $s\le\pi/4$ and 
$b(s)=2\pi$ if $s\ge 3\pi/4$.
Define $p\colon S^1\to S^1\vee S^1$ by
\[
 p(\theta)\coloneqq
 \begin{cases}
  b(\theta)\in S^1_{\text{left}} & 0\le s\le \pi, \\
  b(\theta-\pi) \in S^1_{\text{right}} & \pi\le s\le 2\pi.
 \end{cases}
\]
Then $p$ is smooth.  
Therefore, we have a smooth map 
\begin{equation}\label{eq:concatenation} 
 c\coloneqq p^* \colon  C^{\infty}(S^1\vee S^1,N)\to C^{\infty}(S^1,N)
\end{equation} 
which is called the \textit{concatenation map}.

\begin{prop}\label{prop:concatenation_preserves_metric}
The map $c$ preserves the metrics; that is, $c^*g = g_\vee$. 
\end{prop}

\begin{proof}
For a plot $P\colon U_P\to C^{\infty}(S^1\vee S^1,N)$, 
we define a plot $P_\star \colon U_P\to C^{\infty}(S^1,N)$ ($\star=\text{left},\text{right}$) by $P_\star = (i_\star)^*\circ P$. 
Then
\[
 (\ev_{\theta}\circ(c\circ P))(r) = (\ev_{p(\theta)}\circ P)(r) = P(r)(p(\theta))=
 \begin{cases}
  P(r)_{\text{left}}(b(\theta))  & 0\le\theta\le\pi,\\
  P(r)_{\text{right}}(b(\theta-\pi)) & \pi\le\theta\le 2\pi.
 \end{cases}
\]
Thus for $v,w\in T_rU_P$,
\begin{align*}
 &(c^*g)(P)_r(v,w)= g(c\circ P)_r(v, w) \\
 & \quad = \int_0^{\pi}g_M(\ev_{p(\theta)}\circ P)_r(v,w)d\theta
   +\int_{\pi}^{2\pi}g_M(\ev_{p(\theta)}\circ P)_r(v,w)d\theta\\
  &\quad = \int_{S^1_{\text{left}}}g_M(\ev_{\theta}\circ P_{\text{left}})_r(v,w)d\theta
   +\int_{S^1_{\text{right}}}g_M(\ev_{\theta}\circ P_{\text{right}})_r(v,w)d\theta \\
  &\quad =g(P_{\text{left}})_r(v,w)+g(P_{\text{right}})_r(v,w)=g_{\vee}(P)_r(v,w). 
\end{align*}
Observe that the second equality follows from Lemma \ref{lem:Theta_simplified}. 
\end{proof}

Suppose that $N$ admits a definite weak Riemannian metric $g$ with respect to a generating family $\mathcal{G}$ of the diffeology of $N$. 
Lemma \ref{lem:C} enables us to deduce that the concatenation map \eqref{eq:concatenation} is restricted to the smooth map 
\[
 c \colon  (C^{\infty}(S^1\vee S^1,N), \langle \mathcal{F}_\mathcal{G}^{S^1\vee S^1N}\rangle)\to (C^{\infty}(S^1,N), \langle \mathcal{F}_\mathcal{G}^{S^1N}\rangle). 
\]
The proof of  Lemma \ref{lem:C} implies that that $c\circ P$ is in $\mathcal{F}_\mathcal{G}^{S^1N}$ for each $P \in \mathcal{F}_\mathcal{G}^{S^1\vee S^1N}$.  
Moreover, the proof of  Proposition \ref{prop:concatenation_preserves_metric} allows us to deduce that 
$g_{\vee}(P)_r(v,v) = g(c\circ P)_r(v, v)$ for each plot 
$P \in \mathcal{F}_\mathcal{G}^{S^1\vee S^1N}$ and $v \in T_rU_P$. 
By virtue of Theorem \ref{thm:D'}, we have the following result.

\begin{prop}
The metric $g_\vee$ on  $(C^{\infty}(S^1\vee S^1,N), \langle \mathcal{F}_\mathcal{G}^{S^1\vee S^1N}\rangle)$ is definite. 
\end{prop}

\section{The warped product of Riemannian diffeological spaces}\label{sect:warped_products}
This section introduces the warped product in diffeology. 
Given weak Riemannian diffeological spaces $(X, g_X)$ and $(Y, g_Y)$, we consider $T_2(X)$, $T_2(Y)$ and $T_2(X \times Y)$. 
Let $f \colon X \to \mathbb{R}$ be a positive smooth map, $\pi_1 \colon T_2(X \times Y) \to T_2(X)$ and $\pi_2\colon T_2(X\times Y)\to T_2(Y)$ the projections. 
Then, we see that the map
\[
 g_X\circ\pi_1 + (f \circ \rho \circ \pi_1) \cdot (g_Y\circ\pi_2) \colon T_2(X) \times T_2(Y) \to \mathbb{R} 
\]
is smooth on $T_2(X)\times T_2(Y)$, where 
$\rho \colon T_2(X) \to X$ is the natural map.
Moreover, 
by the universality of the product, 
there exists a smooth map $i \colon T_2(X \times Y) \to T_2(X) \times T_2(Y)$.
We define a map $g_{X \times_f Y} \colon T_2(X \times Y) \to \mathbb{R}$ by
\[
 g_{X \times_f Y} = (g_X \circ \pi_1 + (f \circ \rho \circ \pi_1 )\cdot (g_Y\circ\pi_2)) \circ i.
\]
It is readily seen that, by the definition, $g_{X \times_f Y}$ is a weak Riemannian metric on $X \times Y$. 
We shall call $g_{X \times_f Y}$ the \textit{warped product} of weak Riemannian metrics $g_X$ and $g_Y$ with respect to $f$.

\begin{rem} 
The natural map $i \colon T_2(X \times Y) \to T_2(X) \times T_2(Y)$ mentioned above is a diffeomorphism; see \cite[Proposition 2.2.12]{B24}.
\end{rem}

\begin{prop}\label{prop:warp}
If $g_X$ and $g_Y$ are definite, then so is $g_{X \times_f Y}$. 
\end{prop}

\begin{proof}
Let ${\mathcal G}_X$ and ${\mathcal G}_Y$ be generating families of the diffeologies on $X$ and $Y$, respectively. 
Let $\pi_X \colon X \times Y \to X$ and $\pi_Y \colon X \times Y \to Y$ be the projections. 
Then, the product diffeology on $X \times Y$ is generated by the family
\begin{align*}
 {\mathcal F} \coloneqq \{P \colon U_P \to X \times Y \mid \pi_X \circ P \in {\mathcal G}_X, \pi_Y \circ P \in {\mathcal G}_Y\}.
\end{align*}
This implies that $g_{X \times_f Y}$ is definite with respect to ${\mathcal F}$.
\end{proof}

\begin{ex}
By applying Proposition  \ref{prop:An_induction}, we obtain examples of definite weak Riemannian diffeological spaces. 
\begin{enumerate}[label=(\roman*),leftmargin=*]
\item
	The induction $\boldsymbol{+}\to \R^2$ (Example~\ref{ex:+}) induces a definite weak Riemannian metric on $\boldsymbol{+}$.
\item
	Moreover, consider a pullback diagram 
	\[
	 \xymatrix@C25pt@R15pt{ 
	 Y\times_B X \ar[r] \ar[d]& X \ar[d] \\
	 Y \ar[r] & B
	 }
	\]
	in which $X$ and $Y$ are weak Riemannian diffeological spaces with definite metrics.  
	Then, by combining Proposition  \ref{prop:An_induction} with Proposition \ref{prop:warp}, 
	we see that the pullback $Y\times_B X$ admits a definite weak Riemannian metric.  
\end{enumerate}
\end{ex}

Let $N$ be a weak Riemannian diffeological space.  
We introduce another weak Riemannian metric on  $C^\infty (S^1\vee S^1, N)$ endowed with the functional diffeology.

We recall the smooth maps $i,j\colon S^1\to S^1\vee S^1$ in the pushout diagram (\ref{eq:the_Pushout}), 
where $S^1\vee S^1$ denotes the one point union of two circles; see Section \ref{ss:free_loop_space}. 
Then, we have a diagram of the form 
\begin{equation}\label{eq:A_pullback}
\begin{split}
 \xymatrix@C25pt@R15pt{ 
 C^\infty (S^1\vee S^1, N)\ar[r]^-{\widetilde{l}}&LN\times_N LN  \ar[d] \ar[r]^q & LN \times LN \ar[d]^{(\ev_0,\ev_1)} \\
  & N \ar[r]_-{\Delta} & N\times N, } 
\end{split}
\end{equation}
where the square is the pull-back of the diagonal map $\Delta$, and $\widetilde{l}$ is defined by $\widetilde{l}(\gamma) = (\gamma(*),\gamma\circ i, \gamma\circ j)$.

\begin{lem}\label{lem:l}
The map $\widetilde{l}$ is a well-defined diffeomorphism with respect to the functional diffeology. 
\end{lem}

Since $LN\times_N LN$ is a diffeological subspace of $N\times (LN\times LN)$, 
it follows from Proposition \ref{prop:An_induction} that $C^\infty (S^1\vee S^1, N)$ admits a weak Riemannian metric $g_N\oplus(g\oplus g)$.

\begin{proof}[Proof of Lemma \ref{lem:l}.]
First, 
we see that
$LN \times_N LN = \{(n, \gamma_1, \gamma_2)\mid \gamma_1(0) = \gamma_2(0) = n\} \subset N \times LN \times LN$ 
and 
$Z = \{(\gamma_1, \gamma_2) \mid \gamma_1(0) = \gamma_2(0)\} \subset LN \times LN$ 
are diffeomorphic. 
This is because the smooth maps  
\[
 \pr_2 \times \pr_3 \colon LN \times_N LN \to Z,\quad \pr_2 \times \pr_3(n, \gamma_1, \gamma_2) = (\gamma_1, \gamma_2), \ \text{and} 
\]  
\[
(\ev_0 \circ \pr_1, \id_Z) \colon Z \to LN \times_N LN,\quad (\ev_0 \circ \pr_1, \id_Z)(\gamma_1, \gamma_2) = (\gamma_1(0), \gamma_1, \gamma_2),
\]  
are clearly inverses of each other. 
Under this identification, the map $\widetilde{l}$ is interpreted as the map 
\[
 l=i^*\times j^*\colon C^{\infty}(S^1\vee S^1,N)\xrightarrow{} Z,
 \quad
 \gamma\mapsto(\gamma\circ i,\gamma\circ j).
\]
Since $ i $ and $ j $ are smooth, 
it follows that the maps $i^*$ and $j^*$ are also smooth. 
Therefore, the map $ l = i^* \times j^* $, and hence $\widetilde{l}$, are also smooth.

Moreover, for any $ (\gamma_1, \gamma_2) \in Z $, by the universality of the product, there exists a map  
$\gamma \colon S^1 \vee S^1 \to N$
such that the diagram 
\begin{equation*}
 \xymatrix@C20pt@R15pt{
 S^1 \ar@/_5pt/[rd]_{\gamma_1}\ar[r]_-{i} & S^1 \vee S^1 \ar@{..>}[d]^{\exists! \, \gamma} & \ar[l]^-{j}\ar@/^5pt/[ld]^{\gamma_2} S^1 \\
 & N &
 }
\end{equation*}
commutes. 
Let 
$\nu \colon Z \to C^\infty(S^1 \vee S^1, N)$ 
be the map that assigns to each $ (\gamma_1, \gamma_2) \in Z$ its universal morphism $\gamma $. 
Conversely, given a map $\gamma \colon S^1 \vee S^1 \to N$, by the universal property, 
there exist unique maps $\gamma_1, \gamma_2 \colon S^1 \to N$ such that $\gamma_1 = \gamma \circ i$ and $\gamma_2 = \gamma \circ j$. 
Therefore, we see that $\nu$ and $l$ are inverse mappings of each other. 
We show the smoothness of $\nu$. 
Consider the adjoint $\ad(\nu) \colon  Z\times(S^1 \vee S^1)\to N$ to $\nu$. 
Then, we have a commutative diagram 
\begin{equation*}
 \xymatrix@C20pt@R15pt{
 Z \times (S^1 \coprod S^1) \ar@{^{(}->}[rr] \ar[dd]^{\id_{Z} \times \pi} && LN \times LN \times (S^1 \coprod S^1) \ar[d]^{\cong}\\
                                                   &&  LN \times LN \times S^1 \coprod LN \times LN \times S^1 \ar[d]^{\pr_1 \times \id_{S^1} \coprod\pr_2 \times \id_{S^1}} \\
 Z \times (S^1 \vee S^1) \ar[r]_-{\ad(\nu)}  & N & LN \times S^1 \coprod LN \times S^1. \ar[l]^-{\ev \coprod \ev}
 }
\end{equation*}
Observe that the quotient map $\pi \colon S^1 \coprod S^1 \to S^1 \vee S^1$ is a subduction and then $\id_Z\times \pi$ is also a subduction. 
Thus, $\nu \colon Z \to (C^\infty(S^1 \vee S^1, N), \D_{\mathrm{func}})$ is smooth. %
Hence, the map $l$ is a diffeomorphism with respect to the functional diffeology.
\end{proof}

\begin{rem}\label{rem:thePullback} 
We consider the weak Riemannian metric $g$ on $LN$ described in Section \ref{ss:metric_on_MappingSapces}. 
Observe that we do not require the definiteness of $g$.  
The definition of the metric $g_N\oplus(g\oplus g)$ on $LN\times_N LN$ seems natural but is different from that in \S\ref{ss:free_loop_space}.
When we measure the length of a given curve in $C^{\infty}(S^1\vee S^1,N)$ with the present metric, the length of the trajectory of the attaching point is added twice.
This reflects that the map from $S^1\vee S^1$ passes through the attaching point twice.
\end{rem}

\section{Problems}\label{sect:OpenQ} 
We propose questions on metrics on diffeological spaces, which would be crucial in developing Riemannian diffeology.

\subsection{On $I(M)$ and $(M, \langle \mathcal{G} \rangle)$ in Remark \ref{rem:InfiniteMfd_definiteness}}\label{sect:I(M)}
We use the same notation as in Section \ref{subsection:IDMfd}. 
Suppose that $M$ is of finite dimension. 
Then, each plot of $M$ locally factors through the inverse $\varphi^{-1}$ for some chart $(\varphi, U)$ of $M$.  
Moreover, the inverse is in $\langle \mathcal{G} \rangle$. 
Therefore, we have $\D_{\text{std}}^M \subset \langle \mathcal{G} \rangle$ and then $\D_{\text{std}}^M = \langle \mathcal{G} \rangle$. 
However,  $\varphi^{-1}$ is not in $\langle \mathcal{G} \rangle$ if $M$ is of infinite dimension.

\begin{enumerate}[label={P1.}]
\item
	When is  $\D_{\text{std}}^M = \langle \mathcal{G} \rangle$ for an infinite dimensional manifold $M$?
\end{enumerate}

\subsection{$\D'$ versus $\D_{\text{func}}$}
Let $X$ and $N$ be diffeological spaces and $\mathcal{G}$ a generating family of the diffeology of $N$. 
In Section \ref{sect:D'}, we introduce a subdiffeology $\D'$ of the functional diffeology $\D_\text{func}$ of $C^\infty(X, N)$. 
The diffeology $\D'$ defined by $\mathcal{G}$ with the property (E) in Section \ref{sect:D'} plays a crucial role 
when considering the definiteness of a weak Riemannian metric on $C^\infty(X, Y)$.  
Suppose that $N$ is a Riemannian manifold.
\begin{enumerate}[label={P2.}]
\item
	Is the bijection $l\colon C^\infty (S^1\vee S^1, N)\to LN\times_N LN$ in Lemma \ref{lem:l} a 
	diffeomorphism with respect to the subdiffeology?
\end{enumerate}

\subsection{A Riemannian orbifold}
Let $X$ be an effective orbifold \cite{Wolak} with $n$-dimensional orbifold atlas $\mathcal{U}=\{(\widetilde{U}_i,\Gamma_i,\phi_i)\}$
(called a \emph{defining family} in \cite[Definition 34]{IZKZ}; each $(\widetilde{U}_i,\Gamma_i,\phi_i)$ is called a \emph{local uniformizing  system}).
Suppose $X$ is equipped with a Riemannian metric $g$ \cite[A.1]{Wolak}, 
namely a family $\{g_i\}$ of $\Gamma_i$-invariant Riemannian metrics $g_i$ on $\widetilde{U}_i$ in the usual sense, with obvious compatibility conditions.

Let $X$ be equipped with a diffeology generated by the quotient maps $\phi_i\colon\widetilde{U}_i\to U_i$, as in \cite[Proposiotion 38]{IZKZ}.
Then $X$ is a diffeological orbifold in the sense of \cite[Definition 6]{IZKZ}.

We show that the index set $\D_{\text{res}}$ consisting of the local uniformizing system of $X$ and the Riemannian metric $g$ of $X$ gives rise to a smooth map
\[
 g\colon T_2(X,\D_{\text{res}})\coloneqq\colim_{\D_{\text{res}}}\widehat{T}_2(\widetilde{U}_i)\to\R.
\]
Let $\phi_i\colon\widetilde{U}_i\to X$ ($i=1,2$) be local uniformizing systems such that $U_1\cap U_2\ne\emptyset$.
Then we can find another local uniformizing system $\phi_3\colon\widetilde{U}_3\to U_3$ with $U_3\subset U_1\cap U_2$, 
and embeddings $\psi_{3i}\colon\widetilde{U}_3\hookrightarrow\widetilde{U}_i$ ($i=1,2$) that are equivariant,
in the sense that there exist homomorphisms $\alpha_{3i}\colon\Gamma_3\to\Gamma_i$ with
$\psi_{3i}(h\cdot s)=\alpha_{3i}(h)\cdot\psi_{3i}(s)$ for $s\in\widetilde{U}_3$ and $h\in\Gamma_i$; 
see \cite[p.~201]{Wolak}.
\begin{equation}\label{eq:LUS_compatibility}
\begin{split}
\xymatrix{
 \widetilde{U}_1\ar[rd]_-{\phi_1} & \widetilde{U}_3\ar@{_(->}[l]_-{\psi_{31}}\ar@{^(->}[r]^-{\psi_{32}}\ar[d]^-{\phi_3} & \widetilde{U}_2\ar[ld]^-{\phi_2} \\
  & X &
 }
\end{split}
\end{equation}
Since $g\colon\widehat{T}_2(\widetilde{U}_i)\to\R$ is $\Gamma_i$-invariant and $\psi_{ji}$ is equivariant for each pair $(i,j)$,
the Riemannian metric $g$ of $X$ gives rise to a smooth map $g\colon T_2(X,\D_{\text{res}})\to\R$.

\begin{enumerate}[label={P3.}]
\item
	When is there a lift $\widehat{g}$ of $g$?
	\[
	\xymatrix@C30pt@R15pt{
	T_2(X, \D_{\text{res}}) \ar[r]^-g  \ar[d] & \R.  \\
	T_2(X) \ar@{.>}[ru]_{\widehat{g}}&
	}
	\]
\end{enumerate}

Now we discuss a condition for $g$ to be extended to a smooth map $\widehat{g}\colon T_2(X)\to\R$.

Let $P\colon U_P\to X$ be a plot.
For $r\in U_P$, suppose we have two local factorizations
$P|_{W_i}=\phi_i\circ f_i$ ($i=1,2$),
where $W_i$ is an open neighborhood of $r$ in $U_P$ and $f_i\colon W_i\to\widetilde{U}_i$ is a smooth map.
We may suppose $W_1=W_2$ ($\eqqcolon W$).
We can find $\phi_3\colon\widetilde{U}_3\to U_3$ and $\psi_{3i}\colon\widetilde{U}_3\hookrightarrow\widetilde{U}_i$ ($i=1,2$) as in \eqref{eq:LUS_compatibility}. 
Since $\psi_{3i}$ is an embedding of domains of the same dimensions, we can find $f_{3i}\colon W\to\widetilde{U}_3$ that satisfies $f_i=\psi_{3i}\circ f_{3i}$, after replacing $W$ with a smaller one.
\[
 \xymatrix{
 & W\ar@/_5pt/[ld]_-{f_1}\ar@/^5pt/[rd]^-{f_2}\ar@/_5pt/[d]_-{f_{31}}\ar@/^5pt/[d]^-{f_{32}} & \\
 \widetilde{U}_1\ar[rd]_-{\phi_1} & \widetilde{U}_3\ar@{_(->}[l]_-{\psi_{31}}\ar@{^(->}[r]^-{\psi_{32}}\ar[d]^-{\phi_3} & \widetilde{U}_2\ar[ld]^-{\phi_2} \\
  & X &
 }
\]
Since we have $\phi_3\circ f_{31}=\phi_3\circ f_{32}$ by the commutativity,
for any $x\in W$,
there exists $\gamma_x\in\Gamma_3$ satisfying $f_{31}(x)=\gamma_x\cdot f_{32}(x)$.
If we can choose $\gamma\colon W\to\Gamma_3$ to be smooth, 
then the map $g$ may extend to $T_2(X)$.

\subsubsection*{Acknowledgement}
The authors thank David Miyamoto for valuable comments on the first version of this article and for showing them the generating family in Example \ref{ex:(E)}. 
The authors are grateful to Yoshihisa Miyanishi for discussions with the first authors without which they could not have obtained the refined proof of Proposition \ref{prop:distance_on_Func}. 
The authors would also like to thank the referee for careful reading and valuable comments on a version of this manuscript which led their attention to infinite dimensional manifolds, orbifolds and the category of Riemannian diffeological spaces.

\end{document}